\documentclass{siamltex}
\usepackage{amsfonts,amsmath,amssymb}
\usepackage{mathrsfs,mathtools,mathabx}
\usepackage{enumerate}
\usepackage{xspace,mydef}
\usepackage{hyperref}
\usepackage{esint}
\usepackage{graphicx}
\usepackage{psfrag}
\DeclareMathAlphabet{\mathpzc}{OT1}{pzc}{m}{it}

\usepackage[usenames,dvipsnames]{color}
\newcommand{\EO}[1]{{\color{black}#1}}

\newtheorem{remark}[theorem]{Remark}
\numberwithin{equation}{section}

\newcommand{\calL}{{\mathcal L}}

\newcommand{\calQ}{{\mathcal Q}}

\newcommand{\calJ}{{\mathcal J}}

\newcommand{\pe}{\mathscr{P}}

\newcommand{\ee}{\mathscr{E}}
\newcommand{\calR}{{\mathcal R}}
\newcommand{\calG}{{\mathcal{G}}}
\newcommand{\bmr}{{\boldsymbol r}}
\newcommand{\slope}{\mathfrak{s}}

\newcommand{\bA}{{\mathbf A}}

\def\limn{\lim_{n\to\infty}}
\def\limsupn{\limsup_{n\to\infty}}
\def\liminfn{\liminf_{n\to\infty}}
\def\C{\mathcal{C}}

\title{
A reaction coefficient identification problem for fractional diffusion\thanks{
The research of EO was supported in part by CONICYT through project FONDECYT 11180193. TNTQ gratefully acknowledges support of the University of Goettingen, Goettingen, Lower Saxony, Germany.}}
\author{
Enrique Ot\'arola\thanks{Departamento de Matem\'atica, Universidad T\'ecnica Federico Santa Mar\'ia, Valpara\'iso, Chile (\texttt{enrique.otarola@usm.cl}).}
\and
Tran Nhan Tam Quyen\thanks{Institute for Numerical and Applied Mathematics, University of Goettingen, Goettingen, Germany (\texttt{quyen.tran@uni-goettingen.de}).}
}

\date{Draft version of \today.}

\begin{document}

\maketitle
\begin{abstract}
We analyze a reaction coefficient identification problem for the spectral fractional powers of a symmetric, coercive, linear, elliptic, second--order operator in a bounded domain $\Omega$. We realize fractional diffusion as the Dirichlet-to-Neumann map for a nonuniformly elliptic problem posed on the semi--infinite cylinder $\Omega \times (0,\infty)$. We thus consider an equivalent coefficient identification problem, where the coefficient to be identified appears explicitly. We derive existence of local solutions, optimality conditions, regularity estimates, and a rapid decay of solutions on the extended domain $(0,\infty)$. The latter property suggests a truncation that is suitable for numerical approximation. We thus propose and analyze a fully discrete scheme that discretizes the set of admissible coefficients with piecewise constant functions. The discretization of the state equation relies on the tensorization of a first--degree FEM in $\Omega$ with a suitable $hp$--FEM in the extended dimension. We derive convergence results and obtain, under the assumption that in neighborhood of a local solution the second derivative of the reduced cost functional is coercive, a priori error estimates.
\end{abstract}



\section{Introduction}
\label{sec:introduccion}
In recent times, it has become evident that many of the assumptions that lead to classical models of diffusion are not always satisfactory or even realistic: memory, heterogeneity or a multiscale structure might violate them. In such a scenario, the assumption of locality does not hold and to describe diffusion one needs to resort to nonlocal operators; classical integer order differential operators fail to provide an accurate description. This is specially the case when long range (i.e., nonlocal) interactions are to be taken into consideration. Different models of diffusion have been proposed, fractional diffusion being one of them. An incomplete list of problems where fractional diffusion appears includes finance \cite{MR2064019}, turbulent flow \cite{chen23}, quasi--geostrophic flows models \cite{MR2680400,Kiselev2007}, models of anomalous thermoviscous behaviors \cite{doi:10.1121/1.1646399}, peridynamics \cite{MR3023366,MR3618711}, and imaging sciences \cite{doi:10.1137/070698592,Gatto2015}. Even when having a mathematical model based on fractional diffusion, in a practical setting could occur that such a model is not exact: coefficients or source terms may be subject to uncertainty or unknown. In addition, data or a priori information may be available: we may have a sparse and/or noisy measurement of the state of the system or of an output of interest that we would like to match and/or a priori information of some model coefficients. In such cases, one can resort to the solution of an inverse or control problem to recover such parameters and define a more accurate, data-driven, mathematical model. All these considerations motivate, the need to, on the basis of physical observations, identify coefficients in a fractional diffusion model.

In this work we shall be interested in the analysis of a coefficient identification problem for certain fractional powers of symmetric, coercive, self--adjoint, second order differential operators in bounded domains with homogeneous Dirichlet boundary conditions. To make matters precise, for $d \in \{ 2, 3\}$, we let $\Omega \subset \R^d$ be an open and bounded domain with Lipschitz boundary $\partial \Omega$. We define
\begin{equation}\label{eq:calL}
\calL w := - \DIV_{x'}( A \GRAD_{x'} w ) + qw,
\end{equation}
supplemented with homogeneous Dirichlet boundary conditions. The diffusion matrix $A \in C^{0,1}(\Omega,\GL(\R^d))$ is symmetric and uniformly positive definite and the reaction term $q$ belongs to the following set of admissible coefficients:
\begin{equation}\label{23-2-f2}
\calQ:= \left \{ q \in L^{\infty}(\Omega) ~|~ 0 \leq q(x) \leq \bar{q} \textrm{ a.e. in } \Omega \right \}, \quad \bar q > 0.
\end{equation}
For $s \in (0,1)$, we denote by $\calLs$ the spectral fractional powers of the operator $\calL$.

Given $s \in (0,1)$ and a fixed function $f: \Omega \rightarrow \mathbb{R}$, we shall be concerned with the analysis of the problem of identifying the reaction coefficient $q\in \mathcal{Q}$ in the Dirichlet problem for fractional diffusion
\begin{equation}
\label{eq:fractional_diffusion}
    \calL^s u = f \text{  in  } \Omega
\end{equation}
from the observation data $z_{\delta} \in L^2(\Omega)$ of the exact solution $u^\dag$ satisfying the deterministic noise model
\begin{equation}\label{23-2-f1}
\| u^\dag -z_\delta \|_{L^2(\Omega)} \leq \delta,
\end{equation} 
where $\delta>0$ denotes the level of measurement error. For $f \in \Hsd$, a weak formulation for \eqref{eq:fractional_diffusion} reads: Find $u \in \Hs$ such that
$
\langle \calL^s u, v \rangle = \langle f , v  \rangle
$
for all $v \in \Hs$. This formulation admits a unique solution (see section \ref{sec:notation} for notation and details). 

We mention that the use of $L^2$-observations of the exact state is quite popular in computational inverse problems for elliptic partial differential equations (PDEs). In practice, the observation is measured at certain points of the domain $\Omega$ and an interpolation process is needed to derive distributed observations. Such observation assumptions have been used by many authors; see, for instance \cite{MR1092646,MR852507,0266-5611-26-12-125014,MR2885597,MR2594377,MR3788158,MR1230716,MR2235638}.
As a first step, and in view of several technical difficulties that appears in the analysis and approximation of \eqref{eq:fractional_diffusion}, we assume that observations of $u^\dag$ are available in $\Omega$. We briefly comment about the case of observations being available in a subdomain $\Omega_{\mathrm{obs}}\subset\Omega$ in Remark \ref{rm:extensions} below.

To solve the proposed identification problem in a stable manner,
we will utilize the standard output least squares method with Tikhonov regularization \cite{Banks:1989:ETD:575635,MR2554448,MR1408680,MR2963507}. In fact, for estimating the coefficient $q$ in \eqref{eq:fractional_diffusion} from the observations $z_\delta$ of the exact solution $u^\dag$, we will
invoke
the following cost functional
\begin{equation}\label{eq:functional}
J_{\delta,\rho}(q) := \frac{1}{2} \| U(q) - z_{\delta} \|_{L^2(\Omega)}^2 + \frac{\rho}{2} \| q - q^* \|_{L^2(\Omega)}^2,
\end{equation}
where $U$ denotes the so--called coefficient--to--solution operator, which associates to an element $q \in \calQ$ a unique weak solution $u =: U(q)$ of problem \eqref{eq:fractional_diffusion}. In \eqref{eq:functional}, $q^*$ corresponds to an a priori estimate of the coefficient $q$ to be identified and $\rho>0$ denotes the so--called regularization parameter. We will thus consider a minimizer of the optimization problem
\begin{equation}
\label{eq:min_fractional}
\min_{q \in \calQ} J_{\delta,\rho}(q) 
\end{equation}
as a reconstruction. One of the main difficulties in the analysis and design of solution techniques for problem \eqref{eq:min_fractional} is that $\calL^s$ is a \emph{nonlocal} operator \cite{MR3469920,CS:11,CT:10,CS:07,MR3489634,CDDS:11}. We must immediately notice that the coefficient $q$ to be identified does not appear explicitly neither in the strong nor the weak formulation of problem \eqref{eq:fractional_diffusion}. The coefficient $q$ acts, however, modifying the spectrum of the differential operator $\calL$ and thus the definition of the fractional powers $\calLs$; see definition \eqref{def:second_frac} below. As a consequence, the analysis of the coefficient identification problem \eqref{eq:min_fractional} is intricate.

The mathematical analysis of fractional diffusion has been one of the most studied topics in the past decade \cite{MR3469920,CS:11,CT:10,CS:07,MR2680400,CDDS:11,NOSChina,NOS3,MR3445279,MR3177769}. A breakthrough in the theory, that allows for the localization of $\calL^s$, is due to Caffarelli and Silvestre \cite{CS:07}. When $\Omega = \R^d$ and $\calL=-\Delta$, \ie in the case when $\calL$ coincides with the Laplace operator in $\R^d$, Caffarelli and Silvestre proved that any power $s \in (0,1)$ of the fractional Laplacian $(-\Delta)^s$ can be realized as the Dirichlet--to--Neumann map for an extension problem posed on the upper half--space $\R^{d+1}_+ := \R^d\times(0,\infty)$ \cite{CS:07}. Such an extension problem involves a nonuniformly but \emph{local} elliptic PDE. This result was later adapted in \cite{CT:10,CDDS:11,ST:10} to bounded domains $\Omega$ and more general operators, thereby obtaining a  extension problem posed on the semi--infinite cylinder $\C := \Omega \times (0,\infty)$. The latter corresponds to the following \emph{local} boundary value problem
\begin{equation}\label{eq:alpha_harm_intro}
  \begin{dcases}
   -\DIV\left( y^\alpha \bA \nabla \ue \right) + q y^{\alpha} \ue= 0  & \text{ in } \C, 
    \\
    \ue = 0  &\text{ on } \partial_L \C, \\
    \partial_{\nu^\alpha} \ue = d_s f  &\text{ on } \Omega \times \{0\},
  \end{dcases}
\end{equation}
where $\bA = \diag \{A,1\} \in C^{0,1}(\C, \GL(\R^{d+1}))$ and $\partial_L \C := \partial \Omega \times (0,\infty)$ denotes the lateral boundary of $\C$. In addition, in \eqref{eq:alpha_harm_intro}, $d_s$ denotes a positive normalization constant given by $d_s: = 2^{1-2s} \Gamma(1-s)/\Gamma(s)$ and the parameter $\alpha$ is defined as $\alpha := 1-2s \in (-1,1)$ (cf.\ \cite{CT:10,CS:07,CDDS:11}). 
The conormal exterior derivative of $\ue$ at $\Omega \times \{ 0 \}$ is defined by
\begin{equation}
\label{def:lf}
\partial_{\nu^\alpha} \ue = -\lim_{y \rightarrow 0^+} y^\alpha \ue_y,
\end{equation}
where the limit is understood in the distributional sense. We will call $y \in (0,\infty)$ the \emph{extended variable} and call the dimension $d+1$ in $\R_+^{d+1}$ the \emph{extended dimension}. With the extension $\ue$ at hand, the fractional powers of $\calL$ in \eqref{eq:fractional_diffusion} and 
the Dirichlet--to--Neumann operator of problem \eqref{eq:alpha_harm_intro} are related by
\begin{equation} \label{eq:identity}
  d_s \calL^s u = \partial_{\nu^\alpha} \ue \text{ in } \Omega.
\end{equation}

Motivated by applications to tomography and related techniques, the study of parameter identification problems in PDEs
has received considerable attention over the past 50 years. A rather incomplete list of problems where parameter identification problems appear includes  modern medical imaging modalities, aquifer analysis, geophysical prospecting and pollutant detection. We refer the interested reader to \cite{Banks:1989:ETD:575635,benning_burger_2018,Colton,MR2963507,doi:10.1137/1.9780898717921} for a survey. In particular, the problem of identifying the reaction coefficient in local and elliptic equations has been extensively studied; we refer the reader to \cite{0266-5611-26-12-125014,MR3788158} and references therein. In contrast to these advances, and to the best of our knowledge, this is the first work addressing the study of a reaction coefficient identification problem for fractional diffusion.

We provide a comprehensive treatment for a reaction coefficient identification problem for the spectral fractional powers of a symmetric, coercive, linear, elliptic, second--order operator in a bounded domain $\Omega$. We overcome the nonlocality of fractional diffusion by using the results of Caffarelli and Silvestre \cite{CS:07}. We realize \eqref{eq:fractional_diffusion} by \eqref{eq:alpha_harm_intro} and propose an equivalent identification problem. As a consequence, standard variational techniques are applicable since, in contrast to \eqref{eq:fractional_diffusion}, the reaction coefficient $q$ appears explicitly in the weak formulation of \eqref{eq:alpha_harm_intro}.  This is one of the highlights of our work. We rigorously derive existence and differentiability results, optimality conditions and regularity estimates. We also present a numerical scheme, and prove that there exists a sequence of local minima that converges to a local and exact solution for all values of $s$.

Our presentation is organized as follows. The notation and  functional setting is described in section \ref{sec:notation}, where we also briefly describe, in \S \ref{sub:fracLaplace}, the definition of spectral fractional diffusion and, in \S \ref{sub:CaffarelliSilvestre}, its localization via the Caffarelli--Silvestre extension property. In section \ref{sec:extended} we introduce the \emph{extended} identification problem and prove that is equivalent to \eqref{eq:min_fractional}. In addition, we derive the existence of local solutions, study differentiability properties for the underlying coefficient--to--solution operator, analyze optimality conditions and derive regularity estimates. In section \ref{sec:truncated}, we begin the numerical analysis of our problem. We introduce a truncation of the state equation and a \emph{truncated} identification problem. We derive approximation properties of its solution. In section \ref{subsec:fems} we briefly recall the finite element scheme of \cite{BMNOSS:17} that approximates the solution to \eqref{eq:alpha_harm_intro}. In section \ref{sub:fully_discrete_scheme} we introduce a fully discrete scheme for the truncated identification problem and derive convergence of discrete solutions when the regularization parameter converges to zero. Finally, in section \ref{sub:a_priori}, and under the assumption that in neighborhood of a local solution the second derivative of the reduced cost functional is coercive, we derive convergence results and a priori error estimates for the proposed fully discrete scheme.

\section{Notation and preliminaries}
\label{sec:notation}
We adopt the notation of \cite{NOS,Otarola}. Throughout this work $\Omega$ is a convex polytopal subset of $\mathbb{R}^d$, $d \in \{ 2, 3\}$, with boundary $\partial \Omega$. Besides the semi--infinite cylinder $\C=\Omega \times (0,\infty)$, we introduce the truncated cylinder with base $\Omega$ and height $\Y$ as $\C_{\Y}:= \Omega \times (0,\Y)$ and its lateral boundary $\partial_L\C_{\Y}:= \partial \Omega \times (0,\Y).$ Since the extended variable $y$ will play a special role in the analysis that we will perform, throughout the text, points $x \in \C = \Omega \times (0,\infty) \subset \R^{d+1}$ will be written as 
\[
 x = (x',y), \quad x' \in \Omega\subset \R^d, \quad y \in (0,\infty).
\]

Whenever $X$ is a normed space we denote by $ \| \cdot \|_{X}$ its norm and by $X'$ its dual. For normed spaces $X$ and $Y$ we write $X \hookrightarrow Y$ to indicate continuous embedding.

By $a \lesssim b$ we mean $a \leq Cb$  with a constant $C$ that neither depends 
on $a$, $b$ nor the discretization parameters. The notation $a \sim b$ signifies
$a \lesssim b\lesssim a$. The value of $C$ might change at each occurrence. 

Finally, since we assume $\Omega$ to be convex, in what follows we will make use, without explicit mention, of the following regularity result \cite{Grisvard}:
\begin{equation}
\label{eq:Omega_regular}
\| w \|_{H^2(\Omega)} \lesssim \| \mathcal{L} w \|_{L^2(\Omega)} \quad \forall w \in H^2(\Omega) \cap H^1_0(\Omega).
\end{equation}
\subsection{Fractional powers of elliptic operators}
\label{sub:fracLaplace}
In this section, we invoke spectral theory \cite{BS,Kato} and define the \emph{spectral fractional powers} of the elliptic operator $\calL$. 
To accomplish this task, we begin by noticing that in view of the assumptions on $A$ and $q$, the operator $\calL$ induces the following inner product on $H^1_0(\Omega)$:
\begin{equation}
\label{eq:blfOmega} 
\blfa{\Omega}: H^1_0(\Omega) \times H^1_0(\Omega) \rightarrow \mathbb{R},
\qquad
\blfa{\Omega}(w,v) = \int_\Omega  \left(A \nabla w \cdot \nabla v + qwv \right)  \diff x'.
\end{equation} 
In addition, $\calL: H^1_0(\Omega) \rightarrow H^{-1}(\Omega)$ defined by $u \mapsto  \blfa{\Omega}(u,\cdot)$ is an isomorphism. The eigenvalue problem 
\[
(\lambda,\phi) \in {\mathbb R} \times H^1_0(\Omega) \setminus \{0\}: \quad
 \blfa{\Omega}(\phi,v) = \lambda ( \phi, v)_{L^2(\Omega)}
 \quad \forall v \in H^1_0(\Omega)
\]
has a countable collection of solutions $\{\lambda_k, \varphi_k \}_{k \in \mathbb{N}} \subset \R^+ \times H_0^1(\Omega)$, where $\{\varphi_k\}_{k \in \mathbb{N}}$ is an orthonormal basis of 
$L^2(\Omega)$ and an orthogonal basis of $(H_0^1(\Omega), \blfa{\Omega}(\cdot,\cdot))$. The real sequence of eigenvalues $\{\lambda_k\}_{k \in \mathbb{N}}$ is enumerated in increasing order, counting multiplicities and it accumulates at infinite.

With these ingredients at hand, we define, for $s \ge 0$, the fractional Sobolev space
\begin{equation}
\label{def:Hs}
  \Hs = \left\{ w = \sum_{k=1}^\infty w_k \varphi_k ~\big|~ \| w \|_{\Hs} := \left(
  \sum_{k=1}^{\infty} \lambda_k^s w_k^2 \right)^{\frac12} < \infty \right\}.
\end{equation}

\begin{remark}[the equivalent space $\mathcal{H}^s(\Omega)$]\rm
Let us denote by $\{ \mu_k,\phi_k\}_{k \in \mathbb{N}}$ the eigenpairs of the Dirichlet Laplace operator in the bounded domain $\Omega$. Notice that such a classic operator is obtained upon setting $A \equiv I$ and $q \equiv 0$ in \eqref{eq:calL}. With these eigenpairs at hand, we define the space 
\begin{equation}
  \mathcal{H}^s(\Omega) : = \left\{ w = \sum_{k=1}^\infty w_k \phi_k ~\big|~ \| w \|_{\mathcal{H}^s(\Omega)} := \left(
  \sum_{k=1}^{\infty} \mu_k^s w_k^2 \right)^{\frac12} < \infty \right\}.
   \label{eq:mathcalH}
\end{equation}
In the analysis that follows we will make use and mention when relevant that the space $\mathbb{H}^s(\Omega)$, defined in \eqref{def:Hs}, is equivalent to $\mathcal{H}^s(\Omega)$: for $w \in \Hs$, we have that
\begin{equation}
C_{\dagger}  \| w \|_{\mathcal{H}^s(\Omega)} \leq  \| w \|_{\Hs} \leq C^{\dagger} \| w \|_{\mathcal{H}^s(\Omega)},
\end{equation}
where $C^{\dagger}$ and $C_{\dagger}$ denote positive constants.
\label{rk:mathcalH}
\end{remark}

We denote by $\langle \cdot,\cdot \rangle$ the duality pairing between $\Hs$ and $\Hsd$; $\Hsd$ denotes the dual space of $\Hs$. Through this duality pairing, the definition of the space \eqref{def:Hs} and the norm $ \| \cdot \|_{\Hs}$ can both be extended to $s < 0$. 
By real interpolation between $L^2(\Omega)$ and $H^1_0(\Omega)$, we infer, for $s \in (0,1)$, that $\Hs = [L^2(\Omega), H^1_0(\Omega)]_s$. If $s\in(1,2]$, owing to \eqref{eq:Omega_regular}, we have that $\Hs = H^s(\Omega)\cap H^1_0(\Omega)$ \cite{ShinChan}.  

For $s=1$ and $w  = \sum_{k=1}^\infty w_k \varphi_k \in \mathbb{H}^{1}(\Omega)$, we thus have that $\calL w = \sum_{k=1}^\infty \lambda_k w_k \varphi_k \in \mathbb{H}^{-1}(\Omega)$. For $s \in (0,1)$ and  $w  = \sum_{k=1}^\infty w_k \varphi_k \in \Hs$, the operator $\calL^s$ is defined by
\begin{equation}
 \label{def:second_frac}
\calLs:\Hs \rightarrow \Hsd,
\qquad \calL^s w  
= 
\sum_{k=1}^\infty \lambda_k^{s} w_k \varphi_k.
\end{equation} 

A weak formulation for \eqref{eq:fractional_diffusion} reads as follows: Find $u \in \Hs$ such that
\begin{equation}
\label{eq:fractional_diffusion_weak}
\langle \calL^s u, v \rangle = \langle f, v  \rangle \quad \forall v \in \Hs.
\end{equation}
Given $f \in \Hsd$, problem \eqref{eq:fractional_diffusion_weak} admits a unique weak solution $u \in \Hs$ \cite{CT:10,CDDS:11}. In addition, the  following estimate can be derived \cite{CT:10,CDDS:11}
\begin{equation}
\label{23-2-f3}
\| u \|_{\Hs} \lesssim \| f \|_{\Hsd}.
\end{equation}
%
\subsection{An extension property}
\label{sub:CaffarelliSilvestre}
The Caffarelli--Silvestre extension result requires us to deal with the local but nonuniformly elliptic problem \eqref{eq:alpha_harm_intro} (cf.\ \cite{BCdPS:12,CT:10,CS:07, CDDS:11}). It is thus instrumental to define Lebesgue and Sobolev spaces with the weight $y^{\alpha}$, where $\alpha \in (-1,1)$. If $D \subset \R_{+}^{d+1}$, we define $L^2(y^\alpha,D)$ as the Lebesgue space for the measure $y^{\alpha}\diff x$. We also define the weighted Sobolev space
\[
H^1(y^{\alpha},D) :=
  \left\{ w \in L^2(y^{\alpha},D) ~|~ | \nabla w | \in L^2(y^{\alpha},D) \right\},
\]
where $\nabla w$ is the distributional gradient of $w$. We equip $H^1(y^{\alpha},D)$ with the norm
\begin{equation}
\label{eq:wH1norm}
\| w \|_{H^1(y^{\alpha},D)} =
\left(  \| w \|^2_{L^2(y^{\alpha},D)} + \| \nabla w \|^2_{L^2(y^{\alpha},D)} 
\right)^{\frac{1}{2}}.
\end{equation}
Since $\alpha \in (-1,1)$, the weight $y^\alpha$ belongs to the Muckenhoupt class $A_2(\R^{d+1})$ (cf.\ \cite{Javier,FKS:82,GU,Muckenhoupt,Turesson}). 
This in particular implies that $H^1(y^{\alpha},D)$ with norm \eqref{eq:wH1norm}
is a Hilbert space and $C^{\infty}(D) \cap H^1(y^{\alpha},D)$ is dense in $H^1(y^{\alpha},D)$ 
(cf.~\cite[Proposition 2.1.2, Corollary 2.1.6]{Turesson}, \cite{KO84} and \cite[Theorem~1]{GU}). 

We define the weighted Sobolev space
\begin{equation}
  \label{HL10}
  \HL(y^{\alpha},\C) = \left\{ w \in H^1(y^\alpha,\C) ~|~ w = 0 \textrm{ on } \partial_L \C\right\},
\end{equation}
and immediately notice that $\HL(y^{\alpha},\C)$ can be equivalently defined as \cite{CDDS:11}
\begin{align}
\HL(y^{\alpha},\C) = \big\{ &w: \C \rightarrow \R: \quad w \in H^1(\Omega \times (r,t)) \quad \forall \,\,  0<r<t<\infty,
\nonumber
\\  
&w = 0 \textrm{ on } \partial_L \C, \quad \| \nabla w \|_{L^2(y^{\alpha},\C)} < \infty \big\}.
\label{eq:H1weighted}
\end{align}

As \cite[inequality (2.21)]{NOS} shows, the following \emph{weighted Poincar\'e inequality} holds:
\begin{equation}
\label{Poincare_ineq}
\| w \|_{L^2(y^{\alpha},\C)} \lesssim \| \nabla w \|_{L^2(y^{\alpha},\C)}
\quad \forall w \in \HL(y^{\alpha},\C).
\end{equation}
Consequently, the seminorm on $\HL(y^{\alpha},\C)$ is equivalent to \eqref{eq:wH1norm}.  

For $w \in H^1(y^{\alpha},\C)$, $\tr w$ denotes its trace onto $\Omega \times \{ 0 \}$ 
which satisfies
\begin{equation}
\label{Trace_estimate}
\tr \HL(y^\alpha,\C) = \Hs,
\qquad
  \|\tr w\|_{\Hs} \leq C_{\tr} \| w \|_{\HLn(y^\alpha,\C)};
\end{equation}
see \cite[Proposition 2.5]{NOS}. We notice that, if a function $w $ belongs to $\HL(y^{\alpha},\C)$ then, in view of \eqref{eq:H1weighted}, we have, for $y>0$, that $w(\cdot,y) \in H^{1/2}(\Omega)$.

Define the continuous and coercive bilinear form $\blfa{\C}: \HL(y^{\alpha},\C) \times  \HL(y^{\alpha},\C) \to \R$:
\begin{equation}
\label{eq:blf-a} 
\blfa{\C}(w,\phi)(q) =  \int_\C y^\alpha \left( \bA \nabla w \cdot \nabla \phi + q w \phi \right) \diff x' \diff y.
\end{equation}
We shall simply write $\blfa{\C}(w,\phi)$ or $\blfa{\C}$ when no confusion arises. Notice that $\blfa{\C}$ induces an inner product on $\HL(y^\alpha,\C)$ and an \emph{energy norm}. The latter is defined as follows:
\begin{equation}
\label{eq:norm-C} 
\normC{w}:= \blfa{\C}(w,w)^{1/2} \quad \forall w \in \HL(y^\alpha,\C).
\end{equation}
In view of the assumptions on $\bA$ and $q$, we conclude that $\normC{w} \sim \|\nabla w\|_{L^2(y^\alpha,\C)}$. 

We now present a weak formulation for problem \eqref{eq:alpha_harm_intro}:
\begin{equation}
\label{eq:ue-variational} 
\ue \in  \HL(y^{\alpha},\C): \quad
\blfa{\C}(\ue,\phi)(q) = d_s \langle f,\tr \phi \rangle \qquad \forall \phi \in \HL(y^\alpha,\C).
\end{equation}

The fundamental result of Caffarelli and Silvestre \cite{CS:07}, \cite[Proposition 2.2]{CT:10}, \cite{CDDS:11} is stated bellow.

\begin{proposition}[Caffarelli--Silvestre extension result]\label{22-10-18ct1}
Let $s \in (0,1)$ and $u \in \Hs$ be the solution to \eqref{eq:fractional_diffusion} with $f \in \Hsd$. If $\ue$ solves \eqref{eq:ue-variational}, then
\begin{equation}\label{eq:identity2}
u = \tr \ue \text{ in } \Omega, \qquad  d_s \calL^s u = \partial_{\nu^\alpha} \ue \text{ in } \Omega,
\end{equation} 
where $d_s = 2^{1-2s} \Gamma(1-s)/\Gamma(s)$.
\end{proposition}

The $\Hs$--norm of $u$ and the energy norm $\normC{\cdot}$ of $\ue$ are related by
\begin{equation}
\label{eq:norms_are_equal}
 d_s  \| u \|_{\Hs}^2 = \normC{\ue}^2.
\end{equation}


\section{The extended identification problem}
\label{sec:extended}

In order to analyze problem \eqref{eq:min_fractional} and design a numerical technique to efficiently solve it, we will consider an equivalent minimization problem based on \eqref{eq:ue-variational}: the \emph{extended identification problem}. To describe it, we define the extended coefficient--to--solution operator
\begin{equation}
\label{eq:mapE}
E: \calQ \rightarrow \HL(y^{\alpha},\C), \quad q \mapsto \ue(q),
\end{equation}
which, for a given coefficient $q \in \calQ$ associates to it the unique weak solution $\ue =: E(q) \in \HL(y^{\alpha},\C)$ of problem \eqref{eq:ue-variational}. With the map $E$ at hand, we define, for $\rho > 0$ and $z_{\delta} \in L^2(\Omega)$, the cost functional
\begin{equation}\label{eq:functional_22}
\calJ_{\delta,\rho}(q) := \frac{1}{2} \| \tr E(q) - z_{\delta} \|^2_{L^2(\Omega)} + \frac{\rho}{2} \| q - q^* \|_{L^2(\Omega)}^2.
\end{equation}

The \emph{extended identification problem} thus reads as follows:
\begin{equation}
\label{eq:min_extended}
\min_{q \in \calQ} \calJ_{\delta,\rho}(q).
\end{equation}

The following remark is in order.
\begin{remark}[non--uniqueness]\rm
Due to the lack of strict convexity of the cost functional $\mathcal{J}_{\delta,\rho}$ the uniqueness of a minimizer, when it exists, cannot be guaranteed.
\end{remark}

The previous remark motivates the following definition \cite[Section 4.4]{Tbook}.

\begin{definition}[local solution]\rm
A coefficient $\tilde q$ is said to be a local solution for problem \eqref{eq:min_extended} if there exists $\epsilon > 0$ such that for all $q \in \calQ$ that satisfies $\| q -\tilde q\|_{L^2(\Omega)} < \epsilon$ we have that $\mathcal{J}_{\rho,\delta}(\tilde q) \leq \mathcal{J}_{\rho,\delta}(q)$.
\label{def:local_solution}
\end{definition}

In the following result we state the equivalence between the fractional identification problem \eqref{eq:min_fractional} and the extended one \eqref{eq:min_extended}.

\begin{theorem}[equivalence of \eqref{eq:min_fractional} and \eqref{eq:min_extended}] 
\rm
The coefficient $q^{\ddagger} \in \mathcal{Q}$ is a local solution of \eqref{eq:min_extended} if and only if $q^{\ddagger} \in \mathcal{Q}$ is a local solution of  \eqref{eq:min_fractional}. In addition, we have
\begin{equation}
\label{eq:equivalence}
\tr E(q^{\ddagger}) = U(q^{\ddagger}),
\end{equation}
where $\tr$ is defined as in \S \ref{sub:CaffarelliSilvestre}.
\label{thm:equivalence}
\end{theorem}
\begin{proof}
Let $q^{\ddagger} \in \mathcal{Q}$. In view of  \eqref{eq:identity2} we immediately conclude that $\tr E(q^{\ddagger}) = U(q^{\ddagger})$ and that
\[
\frac{1}{2} \| U(q^{\ddagger}) - z_{\delta} \|^2_{L^2(\Omega)} + \frac{\rho}{2} \| q^{\ddagger} - q^* \|_{L^2(\Omega)}^2.
=
\frac{1}{2} \| \tr E(q^{\ddagger}) - z_{\delta} \|^2_{L^2(\Omega)} + \frac{\rho}{2} \| q^{\ddagger} - q^* \|_{L^2(\Omega)}^2.
\]
Consequently, $J_{\delta,\rho}(q^{\ddagger}) = \calJ_{\delta,\rho}(q^{\ddagger})$. This implies our claim and concludes the proof.
\end{proof}

\subsection{Existence of solutions}

We present the following result.

\begin{theorem}[existence of solutions]
\label{thm:exist}
The regularization problem \eqref{eq:min_extended} has a solution $q_{\delta,\rho}$.
\end{theorem}

\begin{proof}
Let $(q_n)_{n \in \mathbb{N}} \subset \calQ$ be a minimizing sequence for problem \eqref{eq:min_extended}, \ie
\begin{align*}
&\limn \left( \dfrac{1}{2}\|\tr E(q_n) -z_\delta\|^2_{L^2(\Omega)} + \dfrac{\rho}{2}\|q_n-q^*\|^2_{L^2(\Omega)}\right) \\
&~\quad = \inf_{q\in \calQ} \left( \dfrac{1}{2}\|\tr E(q) -z_\delta\|^2_{L^2(\Omega)} + \dfrac{\rho}{2}\|q-q^*\|^2_{L^2(\Omega)}\right);
\end{align*}
such a sequence does exist by the definition of the infimum. We immediately notice that in view of \cite[Theorem 1.7]{MR3014456} we can conclude that $\calQ$ is a weakly$^*$ compact subset of $L^\infty(\Omega)$; see also \cite[Remark 2.1]{quyen18} and \cite[Theorem 3.16]{MR2759829}. Consequently, we deduce the existence of $q_{\delta,\rho}\in \calQ$ and a subsequence $(q_{n_k})_{k \in \mathbb{N}} \subset (q_n)_{n \in \mathbb{N}}$ such that $(q_{n_k})_{k \in \mathbb{N}}$ converges weakly$^*$ to $q_{\delta,\rho}$ in $L^\infty(\Omega)$, \ie
\begin{align}\label{1-3-18-f1}
\lim_{k \rightarrow \infty} \int_\Omega q_{n_k} \chi \diff x' = \int_\Omega q_{\delta,\rho} \chi \diff x' \quad \forall \chi \in L^1(\Omega).
\end{align}

On the other hand, in view of \eqref{23-2-f3}, \eqref{eq:norms_are_equal}, and Remark \ref{rk:mathcalH}, we have that $(E_{n_k})_{k \in \mathbb{N} }$, defined by $E_{n_k} :=E(q_{n_k})$, is bounded in $H^1(y^{\alpha},\C)$. We conclude the existence of $\Theta \in H^1(y^{\alpha},\C)$ and a nonrelabeled subsequence $(E_{n_k})_{k\in \mathbb{N}}$ such that
\begin{equation}
\label{eq:embeddings}
E_{n_k} \rightharpoonup \Theta 
\mbox{  in  }
H^1(y^{\alpha},\C),
\qquad
E_{n_k} \rightharpoonup \Theta 
\mbox{  in  }
L^2(y^{\alpha},\C),
\qquad
k \uparrow \infty.
\end{equation}
In addition, since $\tr: \HL(y^{\alpha},\C) \rightarrow \mathcal{H}^s(\Omega)$ is linear and continuous, the compact embedding $ \mathcal{H}^s(\Omega) \hookrightarrow L^2(\Omega)$ \cite[Theorem 3.27]{Mclean} reveals that
\begin{equation}
\label{eq:embeddings2}
\tr E_{n_k} \rightharpoonup \tr \Theta 
\mbox{  in  } 
\mathcal{H}^s(\Omega),
\qquad
\tr E_{n_k} \rightarrow \tr \Theta 
\mbox{  in  }
L^2(\Omega),
\qquad
k \uparrow \infty.
\end{equation}
We recall that $\mathcal{H}^s(\Omega)$, that is defined in \eqref{eq:mathcalH}, is equivalent to the space $\Hs$; see Remark \ref{rk:mathcalH}.

In what follows we derive that $\Theta = E(q_{\delta,\rho})$. Let  $n_k \in \mathbb{N}$ and $\phi \in \HL(y^\alpha,\C)$. From definitions \eqref{eq:ue-variational} and \eqref{eq:blf-a}, we obtain that
\begin{multline*}
d_s \langle f,\tr \phi \rangle 
= \int_\C y^\alpha \left( \bA \nabla E_{n_k} \cdot \nabla \phi + q_{n_k} E_{n_k} \phi \right) \diff x 
= \int_\C y^\alpha (q_{n_k}-q_{\delta,\rho}) \Theta \phi  \diff x
\\
+
\int_\C y^\alpha \left( \bA \nabla E_{n_k} \cdot \nabla \phi + q_{\delta,\rho}\Theta \phi \right) \diff x
+ 
\int_\C y^\alpha q_{n_k} (E_{n_k} - \Theta)\phi \diff x  =: \mathrm{I}_k + \mathrm{II}_k + \mathrm{III}_k.
\end{multline*}

We proceed to estimate $\textrm{I}_k$. To accomplish this task, we define, a.e. $x'\in \Omega$,  
the function $\chi(x') := \int_0^\infty y^\alpha \Theta(x',y) \phi(x',y) \diff y$ and notice that $\chi \in L^1(\Omega)$. In fact,
\begin{align*}
\| \chi \|_{L^1(\Omega)}
&= \int_\Omega \left|\int_0^\infty y^\alpha \Theta(x',y) \phi(x',y) \diff y\right| \diff x'\\
&\le \int_\C \left| y^\alpha \Theta(x',y) \phi(x',y) \right|\diff x
\le \|\Theta\|_{L^2(y^{\alpha},\C)} \|\phi\|_{L^2(y^{\alpha},\C)} < \infty,
\end{align*}
where we have used that $\Theta \in \HL(y^{\alpha},\C)$. This, together with \eqref{1-3-18-f1}, yields
\begin{align}
\label{28-2-18ct3}
\int_\C y^\alpha q_{n_k}  \Theta \phi \diff x 
&= \int_\Omega q_{n_k} \left(\int_0^\infty y^\alpha \Theta(x',y) \phi(x',y) \diff y\right) \diff x' = \int_{\Omega} q_{n_k} \chi \diff x'\notag\\
&
\rightarrow \int_{\Omega} q_{\delta,\rho} \chi \diff x' 
= \int_\C y^\alpha q_{\delta,\rho}  \Theta \phi \diff x
\end{align}
as $k \uparrow \infty$. As a consequence, when $k \uparrow \infty$, the term $\textrm{I}_k \rightarrow 0$.

Applying directly \eqref{eq:embeddings}, we conclude, as $k \uparrow \infty$, that
\begin{align}\label{28-2-18ct2}
\textrm{II}_k= \int_\C y^\alpha \left( \bA \nabla E_{n_k} \cdot \nabla \phi + q_{\delta,\rho}\Theta \phi \right) \diff x  \rightarrow \int_\C y^\alpha \left( \bA \nabla \Theta \cdot \nabla \phi + q_{\delta,\rho}\Theta \phi \right) \diff x.
\end{align}

To control the term $\mathrm{III}_k$, we proceed as follows. First, notice that
\[
\left| \mathrm{III}_k \right | =
\left| \int_0^{\infty} y^{\alpha} \int_{\Omega} q_{n_k} (x') [ E_{n_k}(x',y) - \Theta(x',y)] \phi(x',y) \diff x' \diff y \right | \leq \overline{q} \int_0^{\infty} \zeta_{n_k}(y) \diff y,
\]
where
\[
\zeta_{n_k}(y):= y^{\alpha} \| E_{n_k} (\cdot,y) - \Theta(\cdot,y) \|_{L^2(\Omega)} \|  \phi (\cdot,y) \|_{L^2(\Omega)}.
\]

Since $E_{n_k} \in \HL(y^{\alpha},\C)$, \eqref{eq:H1weighted} implies that $ \{ E_{n_k}(\cdot,y) \} \subset H^{1/2}(\Omega)$ for a.e. $y \in (0,\infty)$. This and the trace estimate \eqref{Trace_estimate} allow us to conclude that $\{ E_{n_k} (\cdot,y)\}$ converges to $\Theta(\cdot,y)$ in $L^2(\Omega)$ and thus that, a.e. $y \in (0,\infty)$, $\zeta_{n_k}(y) \rightarrow 0$ as $n_k \uparrow 0$. In addition, it can be proved that 
$
( \zeta_{n_k} )_{k \in \mathbb{N}}
$
is uniformly integrable and that, for every $\epsilon > 0$, there exists a set $A \subset (0,\infty)$ of finite measure such that, for all $n_k$, $\int_{A^c} \zeta_{n_k} \diff y < \epsilon$; the latter being a consequence of the exponential decay of $E_{n_k}$ in the extended variable $y$ \cite[Proposition 3.1]{NOS}. We can thus apply the Vitali convergence theorem \cite[Section 6]{Folland} to conclude that
\begin{equation}
\label{eq:III_k}
\left| \int_0^{\infty} y^{\alpha} \int_{\Omega} q_{n_k} (x') [ E_{n_k}(x',y) - \Theta(x',y)] \phi(x',y) \diff x' \diff y \right | \rightarrow 0,
\end{equation}
as $k \uparrow \infty$.

It thus follows from \eqref{28-2-18ct3}--\eqref{eq:III_k} that
\begin{align*}
\int_\C y^\alpha \left( \bA \nabla \Theta \cdot \nabla \phi + q_{\delta,\rho} \Theta \phi \right) \diff x = d_s \langle f,\tr \phi \rangle \quad \forall \phi \in \HL(y^\alpha,\C),
\end{align*}
\ie we have thus proved that $\Theta = E(q_{\delta,\rho})$. 

We now invoke the fact that $\tr E_{n_k} \rightarrow \tr E(q_{\delta,\rho})$ in $L^2(\Omega)$ and that $\| \cdot \|_{L^2(\Omega)}$ is weakly lower semicontinuous to conclude that
\begin{align*}
\mathcal{J}_{\rho,\delta}(q_{\rho,\delta}) & = \dfrac{1}{2}\|\tr E(q_{\delta,\rho}) -z_\delta\|^2_{L^2(\Omega)} + \dfrac{\rho}{2}\|q_{\delta,\rho}-q^*\|^2_{L^2(\Omega)} \\
&\leq \liminf_{k\to\infty} \left( \dfrac{1}{2}\|\tr E_{n_k} -z_\delta\|^2_{L^2(\Omega)} + \dfrac{\rho}{2}\|q_{n_k}-q^*\|^2_{L^2(\Omega)}\right) \\
&= \inf_{q\in \calQ} \left( \dfrac{1}{2}\|\tr E(q) -z_\delta\|^2_{L^2(\Omega)} + \dfrac{\rho}{2}\|q-q^*\|^2_{L^2(\Omega)}\right).
\end{align*}
Consequently, $q_{\delta,\rho}$ is a minimizer for problem \eqref{eq:min_extended}. This concludes the proof.
\end{proof}

\begin{remark}[local solution]\rm
With Definition \ref{def:local_solution} at hand, the results of Theorem \ref{thm:exist} immediately yield the existence of a local solution $q_{\delta,\rho}$ for \eqref{eq:min_extended}.
\label{remark:local_solution}
\end{remark}
\begin{theorem}[existence of solutions]
The fractional regularization problem \eqref{eq:min_fractional} has a solution $q_{\delta,\rho}$.
\label{thm:exist_fractional}
\end{theorem}
\begin{proof}
The result follows directly from the results of Theorems \ref{thm:equivalence} and \ref{thm:exist}.
\end{proof}

\subsection{Optimality conditions}

We begin with a classical result.

\begin{lemma}[first--order optimality condition]
If $q_{\delta,\rho} \in \calQ$ minimizes the functional $\calJ_{\delta,\rho}$, then $q_{\delta,\rho}$ solves the variational inequality 
\begin{equation}
\label{eq:first_order}
( \calJ_{\delta,\rho}'(q_{\delta,\rho}), q - q_{\delta,\rho})_{L^2(\Omega)} \geq 0
\end{equation} 
for every $q \in \calQ$.
\end{lemma}

To explore \eqref{eq:first_order}, in what follows, we derive the Fr\'echet and thus the G\^{a}teaux differentiability of the map $E$. To accomplish this task, we follow standard arguments (see, for instance, \cite{doi:10.1137/S0036139901386375}) and define, for $q \in \calQ$, the map
\begin{equation}
\label{eq:mapPsi}
\Psi:  L^{\infty}(\Omega) \rightarrow \HL(y^{\alpha},\C), \qquad h \mapsto \Psi(h) = \psi,
\end{equation}
where $\psi \in \HL(y^{\alpha},\C)$ solves
\begin{equation}
\label{eq:psi}
\blfa{\C}\left( \psi ,\phi \right )(q) = - \int_{\C} y^{\alpha}  h E(q) \phi \diff x \quad \forall \phi \in \HL(y^{\alpha},\C).
\end{equation}
Since $h \in L^{\infty}(\Omega)$ and $E(q) \in \HL(y^{\alpha},\C )$, a simple application of the Lax--Milgram Lemma reveals that there exist a unique $\psi$ that solves \eqref{eq:psi} together with the estimate
\begin{equation}
\label{eq:estimate_psi}
\| \nabla \psi \|_{L^2(y^{\alpha},\C	)} \lesssim \| h \|_{L^{\infty}(\Omega)} \|f \|_{\Hsd}.
\end{equation}
Consequently, the map $\Psi: L^{\infty}(\Omega) \rightarrow \HL(y^{\alpha},\C)$ is linear and continuous.

\begin{lemma}[first--order Fr\'echet differentiability]\rm
Let $E: \calQ \rightarrow \HL(y^{\alpha},\C)$ be the extended coefficient--to--solution operator defined in \eqref{eq:mapE}.
The map $E$ is first--order Fr\'echet differentiable. In addition, for $q \in \calQ$ and $h \in L^{\infty}(\Omega)$, we have that $E'(q)h = \Psi(h)$, where $\Psi$ is defined as in \eqref{eq:mapPsi}--\eqref{eq:psi}, and that
\begin{equation}
\label{eq:normEp}
\| \nabla E'(q) h \|_{L^2(y^{\alpha},\C)}\lesssim \| f \|_{\Hsd} \| h \|_{L^{\infty}(\Omega)},
\end{equation}
where the hidden constant is independent of $q$ and $h$.
\label{lemma:differentiability}
\end{lemma}
\begin{proof}
Let $q \in \calQ$ and consider $h \in L^{\infty}(\Omega)$ such that $q+h \in \calQ$. In view of \eqref{eq:ue-variational} and the definition of $E$, given by \eqref{eq:mapE}, we arrive at the identity
\begin{equation}
\blfa{\C}\left( E(q),\phi \right )(q) = d_s \langle f,\tr \phi \rangle = \blfa{\C}\left( E(q+h),\phi \right)(q+h)  \qquad \forall \phi \in \HL(y^\alpha,\C).
\end{equation}
This, and the fact that $\psi$ solves \eqref{eq:psi} allow us to conclude that
\begin{multline}
\int_{\C} y^{\alpha} \left[ \bA \GRAD \left(  E(q+h)- E(q) \right) \cdot \GRAD \phi + (q+h) \left(E(q+h) - E(q) \right) \phi \right] \diff x
\\
= - \int_{\C} y^{\alpha} h E(q) \phi \diff x 
= \blfa{\C}\left( \psi,\phi \right )(q)  \qquad \forall \phi \in \HL(y^{\alpha},\C).
\end{multline}
Consequently, we arrive at
\begin{equation*}
\blfa{\C}\left( E(q+h) - E(q) - \Psi(h),\phi \right )(q+h) = - \int_{\C} y^{\alpha} h  \Psi(h)  \phi \diff x
\qquad
\forall \phi \in \HL(y^{\alpha},\C),
\end{equation*} 
which implies that
\begin{align}
\|\nabla \left(  E(q+h)-E(q) - \Psi(h) \right) \|_{L^2(y^{\alpha},\C)} 
& \lesssim 
\| h \|_{L^{\infty}(\Omega)} \| \Psi(h) \|_{L^2(y^{\alpha},\C)}
\nonumber
\\ 
& \lesssim  \| h \|^2_{L^{\infty}(\Omega)} \| f \|_{\Hsd},
\label{eq:step_for_referee}
\end{align}
where, to obtain the last estimate, we have used \eqref{eq:estimate_psi}. 

Since $\Psi: L^{\infty}(\Omega) \rightarrow \HL(y^{\alpha},\C)$ is a linear and continuous map, we have thus obtained that the map $E$ is first--order Fr\'echet differentiable and that $E'(q) h = \Psi (h)$. The estimate \eqref{eq:normEp} follows from \eqref{eq:estimate_psi}. This concludes the proof.
\end{proof}

Define, for $q\in \calQ$ and a.e $x' \in \Omega$, 
\begin{align}\label{eq:q}
\mathtt{e}(q)(x') := \frac{1}{2} \int_0^\infty y^\alpha E(q)(x',y)^2\diff y.
\end{align}
Invoking \eqref{23-2-f3} and \eqref{eq:norms_are_equal}, we obtain, for every $q \in \calQ$, that $\| \mathtt{e}(q)\|_{L^1(\Omega)} \lesssim \| f \|^2_{\Hsd}$.

\begin{remark}[first--order Fr\'echet differentiability] \rm
Notice that, if $\| \mathtt{e}(q) \|_{L^{\infty}(\Omega)} < \infty$, an alternative bound for the right--hand side of \eqref{eq:psi} can be obtained on the basis of basic estimates and the Poincar\'e inequality \eqref{Poincare_ineq}:
\[
\left|  \int_{\C} y^{\alpha} h E(q) \phi \diff x \right| \lesssim \| h \|_{L^2(\Omega)} \| \mathtt{e}(q) \|^{\frac{1}{2}}_{L^{\infty}(\Omega)} \| \nabla \phi \|_{L^2(y^{\alpha},\C)}.
\]
The solution $\psi$ of problem \eqref{eq:psi} thus satisfies that
\begin{equation*}
\label{eq:estimate_psi2}
\| \nabla \psi \|_{L^2(y^{\alpha},\C	)} \lesssim \| h \|_{L^{2}(\Omega)} \| \mathtt{e}(q) \|^{\frac{1}{2}}_{L^{\infty}(\Omega)}.
\end{equation*}
Consequently,
\begin{equation}
\| \nabla E'(q) h\|_{L^2(y^{\alpha},\C)} \lesssim \| h \|_{L^2(\Omega)} \| \mathtt{e}(q) \|^{\frac{1}{2}}_{L^{\infty}(\Omega)},
\label{eq:first_order_bounded}
\end{equation}
where $q \in \mathcal{Q}$ and $h \in L^{\infty}(\Omega)$.
\label{rk:first_order}
\end{remark}

We thus present the following regularity result.

\begin{lemma}[$\| \mathtt{e}(q) \|_{L^{\infty}(\Omega)} < \infty$]
Let $s \in (0,1)$ and $n \in \{2,3\}$. If $f \in L^3(\Omega) \cap \Ws$, then $\mathtt{e}(q) \in L^{\infty}(\Omega)$.
\label{lemma:eqbounded}
\end{lemma}
\begin{proof}
We immediately notice that, by definition, $\mathtt{e}(x') \geq 0$ for a.e. $x' \in \Omega$.

Now, we observe that
\[
\DIV_{x'}( A \GRAD_{x'} \mathtt{e}(q) )
= \int_0^{\infty} y^{\alpha} E(q) \DIV_{x'}( A \GRAD_{x'} E(q)) \diff y +  \int_0^{\infty} y^{\alpha}\GRAD_{x'} E(q) A \GRAD_{x'} E(q) \diff y.
\]
On the other hand, since $E(q)$ solves problem \eqref{eq:alpha_harm_intro}, we have $-\DIV_{x'}( A \GRAD_{x'} E(q) ) + q E(q) =  y^{-\alpha} \partial_y (y^{\alpha} \partial_y E(q))$ a.e. in $\C$. Consequently,
\begin{multline*}
\int_0^{\infty} y^{\alpha} E(q) [-\DIV_{x'}( A \GRAD_{x'} E(q) ) + q E(q)] \diff y 
\\
= \int_0^{\infty} E(q)  \partial_y (y^{\alpha} \partial_y E(q)) \diff y =-
 \lim_{y \rightarrow 0} E(q) y^{\alpha}  \partial_y E(q) -  \int_0^{\infty} y^{\alpha}(\partial_y E(q) )^2 \diff y.
\end{multline*}
We invoke \cite[formula (2.34)]{NOS} to conclude that $-\lim_{y \rightarrow 0} E(q) y^{\alpha}  \partial_y E(q) = u(q) d_s f$,
where $u(q) = \tr E(q)$ solves \eqref{eq:fractional_diffusion}. Thus, $\mathtt{e} = \mathtt{e}(q)$ solves
\begin{equation}
\label{eq:second_order_e}
- \DIV_{x'}( A \GRAD_{x'} \mathtt{e}) + q \mathtt{e} = d_s f u(q) -\int_0^{\infty} y^{\alpha} \nabla  E(q) \mathbf{A} \nabla  E(q) \diff y  := F(x') - G(x').
\end{equation}
Notice that, since $E(q) \in \HL(y^{\alpha},\C)$, we have $G \in L^1(\Omega)$. In addition, since $f \in \Ws$, $u(q) \in \mathbb{H}^{1+s}(\Omega)$. This implies, on the basis of a Sobolev embedding result, that $u(q) \in L^6(\Omega)$. Since $f \in L^3(\Omega)$, we can thus conclude that $F \in L^2(\Omega)$. We now invoke the regularity results of \cite[Theorem 2.7]{NOS} to explore regularity estimates for $G$. In fact, on the basis of the fact that $A \in C^{0,1}(\Omega,\GL(\R^d))$, basic computations reveal that
\[
|\nabla_{x'} G| \lesssim  \int_0^{\infty} y^{\alpha} |\nabla E(q) \nabla_{x'} \nabla E(q)| \diff y + \int_0^{\infty} y^{\alpha} |\nabla E(q) \nabla E(q)| \diff y.
\]
Thus, an application of \cite[Theorem 2.7]{NOS} yields
\begin{align*}
\| \nabla_{x'} G \|_{L^1(\Omega)} & \lesssim 
\| \nabla E(q) \|_{L^2(y^{\alpha},\C)}   \| \nabla_{x'} \nabla E(q) \|_{L^2(y^{\alpha},\C)} + \| \nabla E(q) \|^2_{L^2(y^{\alpha},\C)}
\\
& \lesssim \| f \|_{\Hsd} \| f \|_{\Ws} +  \| f \|^2_{\Hsd} \lesssim \| f \|_{\Hsd} \| f \|_{\Ws}.
\end{align*}
Consequently, $G \in W^{1,1}(\Omega)$. This, combined with a Sobolev embedding result guarantees that $G  \in L^\mu(\Omega)$ for every $\mu \leq d/(d-1)$. Thus, $F - G \in L^\mu(\Omega)$ for every $\mu \leq d/(d-1)$. Since $\Omega$ is convex, we can thus apply standard regularity estimates for problem \eqref{eq:second_order_e} to conclude that $\mathtt{e}(q) \in W^{2,\mu}(\Omega)$ for every $\mu \leq d/(d-1)$. This implies that $\mathtt{e}(q) \in L^{\infty}(\Omega)$ when $d = 2$. 

For $n=3$ we proceed as follows. Notice that $F(x') - G(x') \leq F(x')$ for a.e. $x' \in \Omega$. This implies that $\mathtt{e}$ satisfies
\begin{equation*}
\label{eq:second_order_e_new}
- \DIV_{x'}( A \GRAD_{x'} \mathtt{e}) + q \mathtt{e} \leq F \textrm{ in } \Omega, \quad \mathtt{e} = 0 \textrm{ on } \partial \Omega.
\end{equation*}
We now consider the problem
\begin{equation*}
\label{eq:second_order_varphi}
- \DIV_{x'}( A \GRAD_{x'} \varphi) + q \varphi = F \textrm{ in } \Omega, \quad \varphi = 0 \textrm{ on } \partial \Omega.
\end{equation*}
Since $\Omega$ is convex, elliptic regularity theory reveals that $\varphi \in H^2(\Omega)$ and thus that $\varphi \in L^{\infty}(\Omega)$. Now, notice that $- \DIV_{x'}( A \GRAD_{x'} \mathtt{e}) + q \mathtt{e} \leq F = - \DIV_{x'}( A \GRAD_{x'} \varphi) + q \varphi$ in $\Omega$ and $\mathtt{e} = \varphi = 0$ on $\partial \Omega$. An application of a weak maximum principle \cite[Theorem 8.1]{GT} allows us to conclude that $\mathtt{e} \leq \varphi$ a.e. in $\Omega$. Since $\mathtt{e} \geq 0$ a.e. in $\Omega$, we have thus proved that $\mathtt{e}$ is bounded in $L^{\infty}(\Omega)$. This concludes the proof.
\end{proof}


To describe the variational inequality \eqref{eq:first_order}, we introduce the adjoint variable
 \begin{equation}
\label{eq:extended_adjoint}
\pe \in \HL(y^{\alpha},\C): \quad 
\blfa{\C}\left(w, \pe \right )(q) = d_s \langle  \tr E(q) - z_{\delta}, w \rangle \quad \forall w \in \HL(y^{\alpha},\C).
\end{equation}
With this adjoint state at hand, we define the auxiliary variable
\begin{equation}
\label{eq:frakp}
 \mathfrak{p}(q) = -\frac{1}{d_s}\int_{0}^{\infty} y^{\alpha} E(q) \pe \diff y.
\end{equation}
Notice that, for $q \in \mathcal{Q}$, $\mathfrak{p}(q) \in L^2(\Omega)$ and $ \| \mathfrak{p}(q) \|_{L^{2}(\Omega)} \lesssim \| \mathtt{e}(q) \|^{\frac{1}{2}}_{L^{\infty}(\Omega)}  \| \pe \|_{L^{2}(y^{\alpha},\C)}$.

The optimality conditions \eqref{eq:first_order} in this setting now read as follow.

\begin{theorem}[first--order optimality condition]
If $q_{\delta,\rho} \in \calQ$ minimizes the functional $\calJ_{\delta,\rho}$, then $q_{\delta,\rho}$ solves the variational inequality 
\begin{equation}
\label{eq:first_order:thm}
 \left( \mathfrak{p}(q_{\delta,\rho}) + \rho(q_{\delta,\rho} -q^*), q - q_{\delta,\rho} \right)_{L^2(\Omega)} \geq 0 \quad \forall q \in \calQ,
\end{equation} 
where $\mathfrak{p}(q_{\delta,\rho}) \in L^2(\Omega)$ is defined in \eqref{eq:frakp}.
\label{thm:first_order}
\end{theorem}
\begin{proof}
Since $E$ is differentiable, an application of \eqref{eq:first_order} reveals that
\[
 (\tr E(q_{\delta,\rho}) - z_{\delta},\tr E'(q_{\delta,\rho})(q-q_{\delta,\rho}))_{L^2(\Omega)} + \rho ( q_{\delta,\rho} -q^*, q - q_{\delta,\rho})_{L^2(\Omega)} \geq 0.
\]
Invoking \eqref{eq:psi} and the results of Lemma \ref{lemma:differentiability}, we
obtain that $E'(q_{\delta,\rho})(q-q_{\delta,\rho})$ solves
\begin{equation}
\label{eq:step_oc}
 a_{\C}( E'(q_{\delta,\rho})(q-q_{\delta,\rho}), w)(q_{\delta,\rho}) = -\int_{\C} y^{\alpha} (q-q_{\delta,\rho}) E(q_{\delta,\rho})w \diff x \quad \forall w \in \HL(y^{\alpha},\C).
\end{equation}
By setting $w =  E'(q_{\delta,\rho})(q-q_{\delta,\rho}) \in \HL(y^{\alpha},\C)$ in \eqref{eq:extended_adjoint} and using first \eqref{eq:step_oc} and then \eqref{eq:frakp} we arrive at
\begin{multline*}
  d_s \langle  \tr E(q_{\delta,\rho}) - z_{\delta}, \tr E'(q_{\delta,\rho})(q-q_{\delta,\rho}) \rangle 
  = a_{\C}(E'(q_{\delta,\rho})(q-q_{\delta,\rho}),\pe(q_{\delta,\rho})) (q_{\delta,\rho}) 
  \\
  = -\int_{\C} y^{\alpha} (q-q_{\delta,\rho}) E(q_{\delta,\rho})\pe(q_{\delta,\rho}) \diff x = ( d_s \mathfrak{p}(q_{\delta,\rho}), q - q_{\delta,\rho} )_{L^2(\Omega)}.
\end{multline*}
This concludes the proof.
\end{proof}

In what follows we present higher--order differentiability results for the extended coefficient--to--solution operator $E$.
\begin{lemma}[high--order Fr\'echet differentiability]
\label{higher-order-dif.}
The map $E : \calQ  \to  \HL(y^{\alpha},\C)$ is infinitely Fr\'echet differentiable on the set $\calQ \subset L^{\infty}(\Omega)$. In addition, for $q\in \calQ$ and $m\ge 2$, the action of the m--th Fr\'echet derivative at $q$, $E^{(m)}(q)$, in the direction $(h_1, h_2,...,h_m) \in L^\infty(\Omega)^m$, that is denoted by $\psi = E^{(m)}(q) (h_1, h_2,...,h_m)$, corresponds to the unique solution to 
\begin{equation}
\label{24-4-18f1}
\psi \in \HL(y^{\alpha},\C): \quad 
\blfa{\C}\left( \psi ,\phi \right ) = - \sum_{i=1}^m\int_{\C} y^{\alpha}  h_i \left[ E^{(m-1)}(q)\xi_i\right] \phi \diff x
\quad
\end{equation}
for all $\phi \in \HL(y^{\alpha},\C)$, where $\xi_i := (h_1,...,h_{i-1},h_{i+1},...,h_m)$. Furthermore, 
\begin{equation}
\label{24-4-18f2}
\| \nabla \psi \|_{L^2(y^{\alpha},\C)} \lesssim \| f \|_{\Hsd} \prod_{i=1}^{m}\| h_i \|_{L^{\infty}(\Omega)},
\end{equation}
where the hidden constant is independent of $q$ and $h$.
\end{lemma}
\begin{proof}
On the basis of an induction argument, the proof follows the arguments developed in the proof of Lemma \ref{lemma:differentiability}. For brevity we skip details.
\end{proof}

In order to derive error estimates for the truncated identification problem (Theorem \ref{eq:exp_error_estimate}) and the numerical scheme proposed in section \ref{subsec:error_estimates} (Theorem \ref{rhm:error_estimate}), in what follows we assume that $\mathcal{J}_{\delta,\rho}''$ is coercive in a neighborhood of a local solution: 
If $q_{\delta,\rho}$ denotes a local solution for problem \eqref{eq:min_extended}, then there exists $\epsilon > 0$ such that for every $\hat q$ in the neighborhood $\| \hat q - q_{\delta,\rho} \|_{L^2(\Omega)} \leq \epsilon$, we have that
\begin{equation}
\label{eq:Jcoercive}
\mathcal{J}_{\delta,\rho}''(\hat q)(q,q) \geq \frac{\theta}{2} \| q \|^2_{L^2(\Omega)},
\end{equation}
for all $q \in L^{\infty}(\Omega)$. Since a local solution $\hat q$ belongs to $\mathcal{Q}$ and $\mathcal{J}_{\delta,\rho}$ is Fr\'echet differentiable with respect to the $L^{\infty}(\Omega)$-topology (see Remark \eqref{rk:first_order} and Lemma \ref{higher-order-dif.}) the assumption \eqref{eq:Jcoercive} is well--formulated; see \cite[Assumption 2.20]{MR2536007}. \EO{In the context of inverse problems, we would like to mention \cite[Theorem 3.5]{Ramlau03}, where the author proves that, for a general non-linear ill-posed problem, \eqref{eq:Jcoercive} is satisfied provided a structural source condition is fulfilled. 
Such a source condition has been proven to be fulfilled for the local reaction coefficient identification problem that involves the operator $\mathcal{L}$ ($s=1$) (see, e.g., \cite{EKN89,MR2885597}).}

\subsection{A regularity result}

Define
\begin{equation}
\label{eq:identification_projection}
\Pi_{[0,\bar q]} (q )(x'):=\min\left\{\bar q,\max\left\{0,q(x')\right\}\right\}\mbox{ for all } x'\mbox{ in } \bar \Omega.
\end{equation}
In view of \eqref{eq:first_order:thm}, standard argument allow us to conclude that
\begin{equation}
\label{eq:identification_projection2}
q_{\delta,\rho} = \Pi_{[0,\bar q]} \left( q^{*}  - \frac{1}{\rho} \mathfrak{p}(q_{\delta,\rho})   \right).
\end{equation}

Define, for $q\in \calQ$ and a.e $x' \in \Omega$, 
\begin{align}\label{eq:p}
\mathtt{p}(q)(x') := \frac{1}{2}\int_0^\infty y^\alpha \pe(q)(x',y)^2\diff y.
\end{align}
The same arguments used in the proof of Lemma \ref{lemma:eqbounded} allow us to conclude that $\| \mathtt{p}(q) \|_{L^{\infty}(\Omega)} < \infty$. With this estimate and the result of Lemma \ref{lemma:eqbounded} at hand, we present the following regularity estimate.

\begin{theorem}[regularity estimate]
Let $q_{\delta,\rho}$ be a local solution for \eqref{eq:min_extended}. If $q^{*} \in H^1(\Omega)$, then $q_{\delta,\rho} \in H^1(\Omega)$. In addition, we have that
\begin{multline}
\label{eq:estimate_grad_frakp_thm}
 \|  \nabla_{x'} q_{\delta,\rho}  \| _{L^2(\Omega)} +  \|  \nabla_{x'} \mathfrak{p}(q_{\delta,\rho})  \| _{L^2(\Omega)}  \lesssim  \| \nabla_{x'} q^*  \|_{L^2(\Omega)}
 \\
 +  \| f  \|_{\Hsd}  \|  \mathtt{p}(q_{\delta,\rho})  \|^{\frac{1}{2}}_{L^{\infty}(\Omega)} + ( \| f  \|_{\Hsd} + \| z_{\delta}  \|_{\Hsd} ) \|  \mathtt{e}(q_{\delta,\rho})  \|^{\frac{1}{2}}_{L^{\infty}(\Omega)},
\end{multline}
where the hidden constant is independent of $q_{\delta,\rho}$ and the problem data.
\label{thm:estimate_grad_frakp}
\end{theorem}
\begin{proof}
Notice that, since $E(q_{\delta,\rho})$ and $\pe(q_{\delta,\rho})$ belong to $\HL(y^{\alpha},\C)$, definition \eqref{eq:frakp} implies that
\begin{equation}
\label{eq:grad_frakp}
 \nabla_{x'} \mathfrak{p}(q_{\delta,\rho}) = -\frac{1}{d_s}\int_{0}^{\infty} y^{\alpha}\left( \nabla_{x'}  E(q_{\delta,\rho}) \pe(q_{\delta,\rho}) +  E(q_{\delta,\rho}) \nabla_{x'} \pe(q_{\delta,\rho}) \right) \diff y.
\end{equation}
Similar arguments to the ones developed in Remark \ref{rk:first_order} combined with stability estimates for the problems that $E(q_{\delta,\rho})$ and $\pe(q_{\delta,\rho})$ solve reveal that
\begin{align*}
 \|  \nabla_{x'} \mathfrak{p}(q_{\delta,\rho})  \| _{L^2(\Omega)} 
 & \lesssim  \|  \nabla E(q_{\delta,\rho})  \|_{L^{2}(y^{\alpha},\C)}   \|  \mathtt{p}(q_{\delta,\rho})  \|^{\frac{1}{2}}_{L^{\infty}(\Omega)} 
 \\
& +   \|  \nabla \pe(q_{\delta,\rho})  \|_{L^{2}(y^{\alpha},\C)} \|  \mathtt{e}(q_{\delta,\rho})  \|^{\frac{1}{2}}_{L^{\infty} (\Omega)}
\\ 
& \lesssim  \| f  \|_{\Hsd}  \|  \mathtt{p}(q_{\delta,\rho})  \|^{\frac{1}{2}}_{L^{\infty}(\Omega)} + ( \| f  \|_{\Hsd} + \| z_{\delta}  \|_{\Hsd} ) \|  \mathtt{e}(q_{\delta,\rho})  \|^{\frac{1}{2}}_{L^{\infty}(\Omega)}.
\end{align*}

We now prove that $q_{\delta,\rho} \in H^1(\Omega)$. To accomplish this task, notice that, in view of the assumption $q^{*} \in H^1(\Omega)$, the previous estimate reveals that $q^{*}- \rho^{-1} \mathfrak{p}(q_{\delta,\rho}) \in H^1(\Omega)$. We can thus invoke \cite[Theorem A.1]{KSbook}, which guarantees that, if $\varphi \in H^1(\Omega)$, then $\max \{ \varphi,0  \} \in H^1(\Omega)$, formula \eqref{eq:identification_projection2}, and definition \eqref{eq:identification_projection}, to conclude that $q_{\delta,\rho} \in H^1(\Omega)$ with the estimate
\begin{equation*}
 \|  \nabla_{x'} q_{\delta,\rho}  \| _{L^2(\Omega)} \lesssim  \|  \nabla_{x'} \mathfrak{p}(q_{\delta,\rho})  \| _{L^2(\Omega)} +  \|  \nabla_{x'} q^* \| _{L^2(\Omega)}.
\end{equation*}
This concludes the proof.
\end{proof}

\section{The truncated identification problem}
\label{sec:truncated}
Since the extended identification problem of section \ref{sec:extended} is posed on the semi--infinite cylinder $\C = \Omega \times (0,\infty)$, a first step towards discretization consists in the truncation of the cylinder $\C$ to the bounded domain $\C_{\Y} = \Omega \times (0,\Y)$ and study the effect of truncation.

We begin our analysis by defining the weighted Sobolev space 
\begin{equation*}
  \HL(y^{\alpha},\C_\Y)  = \left\{ w \in H^1(y^{\alpha},\C_\Y): w = 0 \text{ on }
    \partial_L \C_\Y \cup \Omega_{\Y} \right\},
\end{equation*}
where $\Omega_{\Y} := \Omega \times \{ \Y \}$, and the bilinear form $a_\Y: \HL(y^{\alpha},\C_{\Y}) \times \HL(y^{\alpha},\C_{\Y})$ by
\begin{equation*}
\label{a_Y}
  a_\Y(w,\phi)(q) = \int_{\C_\Y} y^{\alpha} \left( \mathbf{A} \GRAD w \cdot \GRAD \phi +  q w \phi \right) \diff x.
\end{equation*}

With this setting at hand, we introduce the coefficient--to--solution operator: 
\begin{equation}
\label{eq:mapH}
 H: \calQ \rightarrow \HL(y^{\alpha},\C_\Y) 
\end{equation}
which, given a coefficient $q$, associates to it the unique solution $v =: H(q)$ of problem
\begin{equation}
\label{eq:truncated_equation}
v \in \HL(y^{\alpha},\C_{\Y}): \quad a_\Y(v,\phi)(q) = d_s \langle f, \tr \phi\rangle \quad \forall \phi \in \HL(y^{\alpha},\C_{\Y}). 
\end{equation}
We also introduce the cost functional
\begin{equation}\label{eq:functional_truncated}
\calR_{\delta,\rho}(q) := \frac{1}{2} \| \tr H(q) - z_{\delta}\|^2_{L^2(\Omega)} + \frac{\rho}{2} \| q - q^* \|_{L^2(\Omega)}^2.
\end{equation}

The \emph{truncated identification problem} thus reads as follows: 
\begin{equation}
\label{eq:min_truncated}
\min_{q \in \calQ} \calR_{\delta,\rho}(q). 
 \end{equation}

To describe first--order optimality conditions we introduce the adjoint variable
 \begin{equation}
\label{eq:truncated_adjoint}
p \in \HL(y^{\alpha},\C_{\Y}): \quad a_{\Y}(w, p)(q) = d_s \langle  \tr H(q) - z_{\delta}, w \rangle \quad \forall w \in \HL(y^{\alpha},\C_{\Y}),
\end{equation}
and define
\begin{equation}
\label{eq:frakr}
 \mathfrak{r}(q) = -\frac{1}{d_s}\int_{0}^{\infty} y^{\alpha} H(q) p \diff y.
\end{equation}
Notice that in \eqref{eq:frakr} we have used the trivial extension by zero for $y > \Y$.

Define, for $q\in \calQ$ and a.e $x' \in \Omega$, 
\begin{align}\label{eq:h}
\mathtt{h}(q)(x') := \frac{1}{2}\int_0^\infty y^\alpha H(q)(x',y)^2\diff y
\end{align}
and
\begin{align}\label{eq:p2}
\mathtt{t}(q)(x') := \frac{1}{2}\int_0^\infty y^\alpha p(q)(x',y)^2\diff y.
\end{align}
Notice that we have used, again, trivial extensions by zero for $y > \Y$. We immediately notice that, an application of the results of Lemma \ref{lemma:eqbounded}, yield $\| \mathtt{h}(q) \|_{L^{\infty}(\Omega)} < \infty$ and $\| \mathtt{t}(q) \|_{L^{\infty}(\Omega)} < \infty$.

The arguments elaborated in the proof of Theorems \ref{thm:exist} and \ref{thm:first_order} allow us to obtain the following results.

\begin{theorem}[existence and first--order optimality condition]
\label{exist_truncated}
The regularization problem \eqref{eq:min_truncated} has a solution $r_{\delta,\rho}$. In addition, if $r_{\delta,\rho}$ minimizes \eqref{eq:min_truncated}, then $r_{\delta,\rho}$ solves the variational inequality
\begin{equation}
\label{eq:first_order_truncated}
 (\mathfrak{r}(r_{\delta,\rho}) + \rho(r_{\delta,\rho} -q^*), q - r_{\delta,\rho})_{L^2(\Omega)} \geq 0 \quad \forall q \in \calQ,
\end{equation} 
where $\mathfrak{r}(r_{\delta,\rho}) \in L^2(\Omega)$ is defined in \eqref{eq:frakr}.
\end{theorem}

\begin{remark}[local solution]\rm
In view of the results of Theorem \ref{exist_truncated} we conclude that problem \eqref{eq:min_truncated} has a local solution $r_{\delta,\rho}$ in the sense of Definition \ref{def:local_solution}.
\end{remark}

The arguments elaborated in the proof of Theorem \ref{thm:estimate_grad_frakp} allow us to immediately arrive at the following regularity estimate.

\begin{theorem}[regularity estimate]
Let $r_{\delta,\rho}$ be a local solution for \eqref{eq:min_truncated}. If $q^{*} \in H^1(\Omega)$, then $r_{\delta,\rho} \in H^1(\Omega)$. In addition, we have that
\begin{multline}
\label{eq:estimate_grad_frakp_thm2}
 \|  \nabla_{x'} r_{\delta,\rho}  \| _{L^2(\Omega)} 
+
 \|  \nabla_{x'} \mathfrak{r}(r_{\delta,\rho})   \| _{L^2(\Omega)}  
 \lesssim   \| \nabla_{x'} q^*  \|_{L^2(\Omega)}
 \\
+  \| f  \|_{\Hsd}  \|  \mathtt{t}(q_{\delta,\rho})  \|^{\frac{1}{2}}_{L^{\infty}(\Omega)} + ( \| f  \|_{\Hsd} + \| z_{\delta}  \|_{\Hsd} ) \|  \mathtt{h}(q_{\delta,\rho})  \|^{\frac{1}{2}}_{L^{\infty}(\Omega)},
\end{multline}
where the hidden constant is independent of $r_{\delta,\rho}$ and the problem data.
\label{thm:estimate_grad_frakp2}
\end{theorem}

\subsection{Auxiliary estimates}
\label{subsec:auxiliary_estimates}

The following error estimates are instrumental for the error analysis that we will perform in section \ref{subsec:error_estimates}.

\begin{lemma}[exponential error estimate I]
Let $q \in \calQ$ and $\Y \geq 1$. If $\ue(q)$ and $v(q)$ denote the solutions to problems \eqref{eq:ue-variational} and \eqref{eq:truncated_equation}, respectively, then
\begin{equation}
\label{eq:Eq-Hq}
 \| \tr(\ue(q) - v(q)) \|_{\Hs} \lesssim \|  \nabla(\ue(q) - v(q) )\|_{L^2(y^{\alpha},\C)} \lesssim e^{-\sqrt{\lambda_1}\Y/4}\| f \|_{\Hsd},
\end{equation}
where $\lambda_1$ corresponds to the first eigenvalue of operator $\calL$. The hidden constant is independent of $\ue$, $v$, $q$, and the problem data.
\label{le:Eq-Hq}
\end{lemma}
\begin{proof}
We invoke the problems that $\ue(q)$ and $v(q)$ solve to arrive at
\[
\int_{\C_{\Y}} y^{\alpha} \left(\bA \nabla (\ue(q)-v(q)) \cdot \nabla \phi + q (\ue(q)-v(q)) \phi\right) \diff x = 0 \quad \forall \phi \in \HL(y^{\alpha},\C_{\Y}).
\]
Notice that we have used that $\phi \in \HL(y^{\alpha},\C_{\Y})$ can be extended by zero to $\C$.
An argument based on C\'ea's Lemma immediately yields
\[
 \| \nabla(\ue(q)-v(q)) \|_{L^2(y^{\alpha},\C_{\Y})} \lesssim \inf \{  \| \nabla(\ue(q)-\ee) \|_{L^2(y^{\alpha},\C_{\Y})}, \ee \in \HL(y^{\alpha},\C_{\Y})\}.
\]
The right--hand side of the previous expression is bounded as in \cite[Lemma 3.3]{NOS}. In fact, \cite[Lemma 3.3]{NOS} provides the estimate
\[
 \| \nabla(\ue(q)-v(q)) \|_{L^2(y^{\alpha},\C_{\Y})} \lesssim e^{-\sqrt{\lambda_1}\Y/4}\| f \|_{\Hsd}.
 \]
This, combined with the fact that \cite[Proposition 3.1]{NOS}
\[
 \| \nabla \ue(q)\|_{L^2(y^{\alpha},\C \setminus \C_{\Y})} \lesssim e^{-\sqrt{\lambda_1}\Y/2} \| f \|_{\Hsd},
 \]
allows us to conclude.
\end{proof}

\begin{lemma}[exponential error estimate II]
Let $q \in \calQ$ and $r \in L^{\infty}(\Omega)$. We thus have the following estimate:
\begin{equation*}
 \left| \mathcal{J}_{\delta,\rho}'(q)r - \mathcal{R}_{\delta,\rho}'(q)r   \right | \lesssim
  e^{-\sqrt{\lambda_1} \Y / 4} \| f\|_{\Hsd}
\left(\| \mathtt{p}(q) \|^{\frac{1}{2}}_{L^{\infty}(\Omega)} + \| \mathtt{h}(q) \|^{\frac{1}{2}}_{L^{\infty}(\Omega)}\right)\|r\|_{L^2(\Omega)},
  \label{eq:exp_derivative}
\end{equation*}
where the hidden constant is independent of $q$, $r$, and the problem data.
The functionals $\mathcal{J}_{\delta,\rho}$ and $\mathcal{R}_{\delta,\rho}$ are defined by \eqref{eq:functional_22} and \eqref{eq:functional_truncated}, respectively.
\label{le:xp_derivative}
\end{lemma}
\begin{proof}
First, notice that, in view of definitions \eqref{eq:frakp} and \eqref{eq:frakr}, we have that
\begin{align*}
\mathcal{J}_{\delta,\rho}'(q)r - \mathcal{R}_{\delta,\rho}'(q)r   =  (\mathfrak{p}(q) - \mathfrak{r}(q), r)_{L^2(\Omega)}.
\end{align*}
It thus suffices to bound the term $\| \mathfrak{p}(q) - \mathfrak{r}(q)\|_{L^2(\Omega)}$. In fact, we have that
\begin{multline*}
\| \mathfrak{p}(q) -\mathfrak{r}(q) \|_{L^2(\Omega)} = \frac{1}{d_s} \left \| \int_0^{\infty} y^{\alpha}\left[E(q)\pe(q) - H(q)p(q)\right] \diff y \right \|_{L^2(\Omega)}
\\
= \frac{1}{d_s} \left \| \int_0^{\infty} y^{\alpha}\left[ \left( E(q) - H(q) \right)\pe(q) - H(q) ( p(q) - \pe(q))\right] \diff y \right \|_{L^2(\Omega)}.
\end{multline*}
Standard estimates allow us to arrive at
\begin{multline*}
\| \mathfrak{p}(q) -\mathfrak{r}(q) \|_{L^2(\Omega)} \lesssim \| \mathtt{p}(q) \|^{\frac{1}{2}}_{L^{\infty}(\Omega)} \| E(q) - H(q) \|_{L^2(y^{\alpha},\C)} 
\\
+ \| \mathtt{h}(q) \|^{\frac{1}{2}}_{L^{\infty}(\Omega)}\| p(q) - \pe(q) \|_{L^2(y^{\alpha},\C)}.
\end{multline*}
This, in view of the Poicar\'e inequality \eqref{Poincare_ineq}, the exponential error estimate \eqref{eq:Eq-Hq}, and the stability estimate for the problem that $p(q) - \pe(q)$ solve allow us to conclude
\begin{multline}
\| \mathfrak{p}(q) -\mathfrak{r}(q) \|_{L^2(\Omega)}
 \lesssim e^{-\sqrt{\lambda_1} \Y / 4} 
\| \mathtt{p}(q) \|^{\frac{1}{2}}_{L^{\infty}(\Omega)}
 \| f\|_{\Hsd}
\\
+
\| \mathtt{h}(q) \|^{\frac{1}{2}}_{L^{\infty}(\Omega)} \| \tr (H(q) - E(q))\|_{\Hsd},
\end{multline}
where $( \lambda_1 )^{\frac{1}{2}} = ( \lambda_1(q) )^{\frac{1}{2}}$ denotes the first eigenvalue of the operator $\calL$.
Applying \eqref{eq:Eq-Hq}, again, we conclude that
\begin{equation}
\label{eq:aux_exp}
\| \mathfrak{p}(q) -\mathfrak{r}(q) \|_{L^2(\Omega)}
 \lesssim e^{-\sqrt{\lambda_1}  \Y / 4} \| f\|_{\Hsd} 
\left( \| \mathtt{p}(q) \|^{\frac{1}{2}}_{L^{\infty}(\Omega)} + \| \mathtt{h}(q) \|^{\frac{1}{2}}_{L^{\infty}(\Omega)} \right).
\end{equation}
This implies the desired estimate and concludes the proof.
\end{proof}


\begin{lemma}[stability estimate]
Let $q,r \in \calQ$ and $E$ be the extended coefficient--to--solution operator defined in \eqref{eq:mapE} We thus have the following error estimate:
\begin{equation*}
\| \tr(\ue(q) - \ue(r)) \|_{\Hs}
 \lesssim
\|  \nabla(\ue(q) - \ue(r) )\|_{L^2(y^{\alpha},\C)} 
\\
\lesssim \|\mathtt{e}(r)\|^{\frac{1}{2}}_{L^\infty(\Omega)} \| q- r\|_{L^2(\Omega)},
\end{equation*}
where $\ue(q) = E(q)$ and $\ue(r) = E(r)$ and the hidden constant is independent of $E$, $q$, $r$, and the problem data.
\label{le:stab_for_gradient}
\end{lemma}
\begin{proof}
Since $\ue(q) - \ue(r) \in \HL(y^{\alpha},\C)$, the first estimate follows immediately from the trace estimate \eqref{Trace_estimate}. The remaining estimate follows upon exploiting the problems that $\ue(q)$ and $\ue(r)$ solve. In fact, notice that $\ue(q)-\ue(r)$ satisfies
\[
 \int_{\C} y^{\alpha} \left[ \bA \nabla (\ue(q)-\ue(r)) \cdot \nabla \phi + q (\ue(q)-\ue(r)) \phi\right] \diff x + \int_{\C}y^{\alpha} (q-r)\ue(r)\phi \diff x=0
\]
for all $\phi \in \HL(y^{\alpha},\C)$. We can thus set $\phi= \ue(q) - \ue(r) \in \HL(y^{\alpha},\C)$ and obtain, on the basis of \eqref{eq:norm-C}, Fubini's theorem, and the Cauchy--Schwarz inequality that 
\begin{align*}
& \| \nabla(\ue(q)-\ue(r)) \|_{L^2(y^{\alpha},\C)}^2 
\\
& 
\lesssim \| q- r\|_{L^2(\Omega)} \left \| \int_0^\infty y^\alpha |\ue(r)(x',y)| |\ue(q)(x',y)-\ue(r)(x',y)|\diff y \right\|_{L^2(\Omega)} \\
& \le \| q- r\|_{L^2(\Omega)} \left(\int_\Omega \mathtt{e}(r)(x') \left[\int_0^\infty y^\alpha |\ue(q)(x',y)-\ue(r)(x',y)|^2 \diff y\right] \diff x'\right)^{\frac{1}{2}},
\end{align*}
where, we recall that $\mathtt{e}$ is defined in \eqref{eq:q}. Thus, 
\[
 \| \nabla(\ue(q)-\ue(r)) \|_{L^2(y^{\alpha},\C)}^2 \lesssim \| q- r\|_{L^2(\Omega)} 
\|\mathtt{e}(r)\|_{L^\infty(\Omega)}^{\frac{1}{2}}
 \| \ue(q)-\ue(r) \|_{L^2(y^{\alpha},\C)},
\]
which, in view of the Poincar\'e inequality \eqref{Poincare_ineq}, allows us to arrive at the desired estimate.
\end{proof}

\subsection{Convergence}
\label{subsec:convergence}

The following result
is important since guarantees that every strict local minimum of problem \eqref{eq:min_extended} can be approximated by local minima of \eqref{eq:min_truncated}. 

\begin{theorem}[convergence result]
Let $q_{\delta,\rho} \in \calQ$ be a strict local minimum of \eqref{eq:min_extended}. Then, there exists a sequence $\{ r_{\delta,\rho} \}$ of local minima for \eqref{eq:min_truncated} such that
\begin{equation}
\label{eq:convergence}
 \{ r_{\delta,\rho} \} \rightarrow q_{\delta,\rho}
 \end{equation}
 as the truncation parameter $\Y \uparrow \infty$.
\end{theorem}
\begin{proof}
Since $q_{\delta,\rho} \in \calQ$ is a strict local minimum of \eqref{eq:min_extended}, we conclude the existence of $\epsilon > 0$ such that $q_{\delta,\rho}$ is the unique strict solution to the following problem:
\begin{equation}
\min_{q \in \calQ(q_{\delta,\rho})} \calJ_{\delta,\rho}(q), \qquad \calQ(q_{\delta,\rho}) = \{ q \in \calQ: \| q -  q_{\delta,\rho}\|_{L^2(\Omega)} \leq \epsilon\}.
\label{eq:aux_problem}
\end{equation}
Let us now consider the following truncated optimization problem over $\calQ(q_{\delta,\rho})$:
\begin{equation}
\label{eq:min_truncated_epsilon}
\min_{r \in \calQ(q_{\delta,\rho})} \calR_{\delta,\rho}(r). 
\end{equation}
It is immediate that $\calQ(q_{\delta,\rho}) \neq \emptyset$. This implies that \eqref{eq:min_truncated_epsilon} has at least one solution. Let $r_{\delta,\rho} := r_{\delta,\rho} (\Y)$ be a solution to \eqref{eq:min_truncated_epsilon} for $\Y \geq \Y_0 > 1$. Notice that $\{ r_{\delta,\rho}\}$ is a bounded sequence in $L^{\infty}(\Omega)$. As in the proof of Theorem \ref{thm:exist} we deduce the existence of a non--relabeled subsequence $\{ r_{\delta,\rho}\}$ that converges weakly* to $\tilde q$ in $L^{\infty}(\Omega)$. In addition, notice that the arguments elaborated in the proof of Theorem \ref{thm:exist} combined with the results of Lemma \ref{le:Eq-Hq} allow us to arrive at
\[
\calJ_{\delta,\rho} (\tilde q) 
\leq 
\liminf_{\Y \rightarrow \infty} \calR_{\delta,\rho} (r_{\delta,\rho}). 
\]
We can thus conclude that
\begin{align*}
\calJ_{\delta,\rho} (q_{\delta,\rho}) 
& \leq 
\calJ_{\delta,\rho} (\tilde q) 
\,\leq
\liminf_{\Y \rightarrow \infty} \calR_{\delta,\rho} ( r_{\delta,\rho})
\leq 
\limsup_{\Y \rightarrow \infty} \calR_{\delta,\rho} ( r_{\delta,\rho})
\\
& \leq \limsup_{\Y \rightarrow \infty} \calR_{\delta,\rho} (q_{\delta,\rho}) =  \calJ_{\delta,\rho} (q_{\delta,\rho}).
\end{align*}
Since \eqref{eq:aux_problem} has a unique solution, we can thus conclude that the sequence $\{ r_{\delta,\rho} \}$ converges strongly in $L^2(\Omega)$ to $q_{\delta,\rho}$: $\| q_{\delta,\rho} - r_{\delta,\rho} \|_{L^2(\Omega)} \rightarrow 0$ as $\Y \uparrow \infty$. This implies that the constraint $r_{\delta,\rho} \in \mathcal{Q}(q_{\delta,\rho})$ is not active when $\Y$ is sufficiently large. Consequently, $r_{\delta,\rho}$ is a local solution of problem \eqref{eq:min_truncated}. This concludes the proof.
\end{proof}

\subsection{Error estimates}
\label{subsec:error_estimates}

The next result shows how $\{ r_{\delta,\rho} \}$, a sequence of local minima of \eqref{eq:min_truncated}, approximates a local solution $q_{\delta,\rho}$ of problem \eqref{eq:min_extended}.

\begin{theorem}[exponential error estimate]
Let $q_{\delta,\rho}$ be a local solution of problem \eqref{eq:min_extended}. If the assumption \eqref{eq:Jcoercive} holds and $\{ r_{\delta,\rho} \}$ denotes a sequence of local minima of \eqref{eq:min_truncated} that converges to $q_{\delta,\rho}$ as $\Y \uparrow \infty$ in $L^2(\Omega)$, then 
\begin{equation}
\label{eq:exp_error_estimate}
 \|q_{\delta,\rho} - r_{\delta,\rho} \|_{L^2(\Omega)} \lesssim e^{-\kappa \Y / 4} \| f\|_{\Hsd} \left( \| \mathtt{p}(q) \|^{\frac{1}{2}}_{L^{\infty}(\Omega)} + \| \mathtt{h}(q) \|^{\frac{1}{2}}_{L^{\infty}(\Omega)} \right),
\end{equation}
where $\kappa = ( \lambda_1(r_{\delta,\rho}) )^{\frac{1}{2}}$ and $ \lambda_1(r_{\delta,\rho})$ denotes the first eigenvalue of the operator $\calL$ with $q$ replaced by $r_{\delta,\rho}$. The hidden constant is independent of $q_{\delta,\rho}$, $r_{\delta,\rho}$, and the problem data.
\label{thm:exponential_error_estimate}
\end{theorem}

\begin{proof} 
In view of \eqref{eq:Jcoercive}, we choose $\epsilon >0$ sufficiently small such that, for $\hat q$ in the neighborhood $\|\hat q - q_{\delta,\rho} \|_{L^2(\Omega)} \leq \epsilon$, the following estimate holds:
\begin{equation}
  \mathcal{J}_{\delta,\rho}''(\hat q)(q,q) \geq \frac{\theta}{2} \| q \|^2_{L^2(\Omega)} \quad \forall q \in L^{\infty}(\Omega).
  \label{eq:Jcoercive_aux}
\end{equation}

Since the sequence $\{ r_{\delta,\rho} \}$ converges to $q_{\delta,\rho}$ in $L^2(\Omega)$ as $\Y \uparrow \infty$, we deduce the existence of $\Y_0$ such that, for $\Y \geq \Y_0$, we have $\| q_{\delta,\rho} - r_{\delta,\rho}  \|_{L^2(\Omega)} \leq \epsilon$. We are thus able to set, for $\Y \geq \Y_0$ and $\zeta \in [0,1]$, $q = r_{\delta,\rho} - q_{\delta,\rho}$ and $\hat q = \zeta q_{\delta,\rho} + (1-\zeta) r_{\delta,\rho}$ and invoke the estimate \eqref{eq:Jcoercive_aux} to obtain that
\begin{align*}
\frac{\theta}{2} \| r_{\delta,\rho} - q_{\delta,\rho} \|^2_{L^2(\Omega)} & \leq \mathcal{J}_{\delta,\rho}''(\hat q)(r_{\delta,\rho} - q_{\delta,\rho},r_{\delta,\rho} - q_{\delta,\rho}) 
\\
&  =\mathcal{J}_{\delta,\rho}'(r_{\delta,\rho})(r_{\delta,\rho} - q_{\delta,\rho} ) - \mathcal{J}_{\delta,\rho}'(q_{\delta,\rho})(r_{\delta,\rho} - q_{\delta,\rho}).
\end{align*}
Notice that $\| \hat q - q_{\delta,\rho}\|_{L^2(\Omega)} = (1- \zeta) \|q_{\delta,\rho} - r_{\delta,\rho} \|_{L^2(\Omega)} \leq \epsilon$.

In view of \eqref{eq:first_order} we immediately conclude that $-\mathcal{J}_{\delta,\rho}'(q_{\delta,\rho})(r_{\delta,\rho} - q_{\delta,\rho} ) \leq 0$. On the other hand, \eqref{eq:first_order_truncated} reveals that $-\mathcal{R}_{\delta,\rho}'(r_{\delta,\rho})(r_{\delta,\rho}-q_{\delta,\rho}) \geq 0$. Consequently,
\[
 \frac{\theta}{2} \| r_{\delta,\rho} - q_{\delta,\rho} \|^2_{L^2(\Omega)} \leq (\mathcal{J}_{\delta,\rho}'(r_{\delta,\rho}),r_{\delta,\rho} - q_{\delta,\rho} )_{L^2(\Omega)}
 -
 (\mathcal{R}_{\delta,\rho}'(r_{\delta,\rho}),r_{\delta,\rho} - q_{\delta,\rho} )_{L^2(\Omega)}.
\]
We now exploit that $\mathcal{J}_{\delta,\rho}'(r_{\delta,\rho}) = \mathfrak{p}(r_{\delta,\rho}) + \rho(r_{\delta,\rho}-q^{*})$ and that $\mathcal{R}_{\delta,\rho}'(r_{\delta,\rho}) = \mathfrak{r}(r_{\delta,\rho}) + \rho(r_{\delta,\rho}-q^{*})$ to obtain the following estimate
\begin{multline*}
 \frac{\theta}{2} \| r_{\delta,\rho} - q_{\delta,\rho} \|^2_{L^2(\Omega)} \leq (\mathfrak{p}(r_{\delta,\rho}) + \rho(r_{\delta,\rho}-q^{*}),r_{\delta,\rho} - q_{\delta,\rho} )_{L^2(\Omega)}
 \\
 -
 (\mathfrak{r}(r_{\delta,\rho}) + \rho(r_{\delta,\rho}-q^{*}),r_{\delta,\rho} - q_{\delta,\rho} )_{L^2(\Omega)} = (\mathfrak{p}(r_{\delta,\rho}) -\mathfrak{r}(r_{\delta,\rho}) ,r_{\delta,\rho} - q_{\delta,\rho} )_{L^2(\Omega)}.
\end{multline*}
The control of the term $(\mathfrak{p}(r_{\delta,\rho}) -\mathfrak{r}(r_{\delta,\rho}) ,r_{\delta,\rho} - q_{\delta,\rho} )_{L^2(\Omega)}$ follows from Lemma \ref{le:xp_derivative}. In fact, we have that
\[
 \| \mathfrak{p}(r_{\delta,\rho}) -\mathfrak{r}(r_{\delta,\rho}) \|_{L^2(\Omega)} \lesssim e^{-\kappa \Y/4} \| f\|_{\Hsd} \left( \| \mathtt{p}(q) \|^{\frac{1}{2}}_{L^{\infty}(\Omega)} + \| \mathtt{h}(q) \|^{\frac{1}{2}}_{L^{\infty}(\Omega)} \right).
\]
This concludes the proof.
\end{proof} 

\begin{theorem}[exponential error estimate]
Let $q_{\delta,\rho}$ be a local solution of problem \eqref{eq:min_extended}. If the assumption \eqref{eq:Jcoercive} holds and $\{ r_{\delta,\rho} \}$ denotes a sequence of local minima of \eqref{eq:min_truncated} that converges to $q_{\delta,\rho}$ as $\Y \uparrow \infty$ in $L^2(\Omega)$, then 
\begin{multline}
\label{eq:exp_error_estimate2}
 \| \nabla( \ue( q_{\delta,\rho}) - v(r_{\delta,\rho})) \|_{L^2(y^{\alpha},\C)} 
 \\
 \lesssim e^{-\kappa \Y / 4}  \| f\|_{\Hsd} \left( 1+ \|\mathtt{e}(r_{\delta,\rho})\|^{\frac{1}{2}}_{L^\infty(\Omega)}
 \left(\| \mathtt{p}(q) \|^{\frac{1}{2}}_{L^{\infty}(\Omega)} + \| \mathtt{h}(q) \|^{\frac{1}{2}}_{L^{\infty}(\Omega) } \right)  \right),
\end{multline}
where $\kappa = ( \lambda_1(r_{\delta,\rho}) )^{\frac{1}{2}}$ and $ \lambda_1(r_{\delta,\rho})$ denotes the first eigenvalue of the operator $\calL$ with $q$ replaced by $r_{\delta,\rho}$. The hidden constant is independent of $q_{\delta,\rho}$, $r_{\delta,\rho}$, and the problem data.
\label{thm:exponential_error_estimate2}
\end{theorem}
\begin{proof}
To derive the estimate \eqref{eq:exp_error_estimate2} we proceed as follows:
\begin{multline}
  \| \nabla( \ue( q_{\delta,\rho}) - v(r_{\delta,\rho})) \|_{L^2(y^{\alpha},\C)} \leq   \| \nabla( \ue( q_{\delta,\rho}) - \ue(r_{\delta,\rho})) \|_{L^2(y^{\alpha},\C)}  
  \\
  +   \| \nabla( \ue( r_{\delta,\rho}) - v(r_{\delta,\rho})) \|_{L^2(y^{\alpha},\C)} = \textrm{I} + \textrm{II}.
\end{multline}
The control of the term $\mathrm{I}$ follows from applying, first, the estimate of Lemma \ref{le:stab_for_gradient} and then 
\eqref{eq:exp_error_estimate}. In fact, we have that
\begin{align*}
 | \textrm{I} |
 & \lesssim \|\mathtt{e}(r_{\delta,\rho})\|^{\frac{1}{2}}_{L^\infty(\Omega)} \| q_{\delta,\rho} - r_{\delta,\rho}\|_{L^2(\Omega)}
 \\
 & \lesssim e^{-\kappa \Y / 4} \|\mathtt{e}(r_{\delta,\rho})\|^{\frac{1}{2}}_{L^\infty(\Omega)} \| f\|_{\Hsd} 
  \left( \| \mathtt{p}(q) \|^{\frac{1}{2}}_{L^{\infty}(\Omega)} + \| \mathtt{h}(q) \|^{\frac{1}{2}}_{L^{\infty}(\Omega)} \right).
\end{align*}
The estimate for the second term $\mathrm{II}$ follows from \eqref{eq:Eq-Hq}:
\[
 | \textrm{II} | \lesssim e^{-\kappa \Y / 4} \| f \|_{\Hsd}.
\]
The collection of the estimates for $\mathrm{I}$ and $\mathrm{II}$ yield the desired result.
 \end{proof}

\section{A discretization scheme}
\label{sec:discretization}

In this section we present a fully discrete scheme that approximate solutions to the fractional identification problem \eqref{eq:min_fractional}. In view of the localization results of Theorem \ref{thm:equivalence} and the exponential error estimates of Theorems \ref{thm:exponential_error_estimate} and \ref{thm:exponential_error_estimate2}, in what follows, we design and analyze an efficient solution technique to solve the truncated identification problem \eqref{eq:min_truncated}. We begin by introducing ingredients for a suitable finite element discretization of the truncated equation \eqref{eq:truncated_equation}.

\subsection{Finite element methods}
\label{subsec:fems}
We follow \cite{BMNOSS:17} and introduce a finite element technique that is based on the tensorization of a first--degree FEM in $\Omega$ with a suitable $hp$--FEM on the extended domain $(0,\Y)$. The scheme achieves log--linear complexity with respect to the number of degrees of freedom in the domain $\Omega$. To present it, we first introduce, on the extended interval $[0,\Y]$, the following geometric meshes with $M$ elements and grading factor $\sigma \in (0,1)$: 
\begin{equation}
\label{eq:geometric_mesh}
\calG^M_{\sigma} = \{ I_m \}_{m=1}^M, \quad I_1 = [0,\Y \sigma^{M-1}], \quad I_i = [\Y \sigma^{M-i+1},\Y \sigma^{M-i}],
\end{equation}
with $i \in \{ 2,\ldots,M \}$. The main motivation for considering the meshes $\calG^M_{\sigma}$, that are refined towards $y=0$, is to compensate the rather singular behavior of $\ue$, solution to problem \eqref{eq:ue-variational}, as $y \downarrow 0$; see \cite[Theorem 2.7]{NOS} and \cite[Theorem 4.7]{BMNOSS:17}. On these meshes, we consider a \emph{linear degree vector} $\bmr =(r_1,\dots,r_M) \in \mathbb{N}^M$ with slope $\slope$: $r_i := 1 + \lceil \slope (i-1) \rceil$, where $i=1,2,...,M$. We thus introduce the finite element space
\[
S^\bmr((0,\Y),\calG_{\sigma}^M) 
= 
\left \{ v_M \in C[0,\Y]: v_M|_{I_m} \in \mathbb{P}_{r_m}(I_m), I_m \in \calG_{\sigma}^M, m =1, \dots, M \right \},
\]
and the subspace of $S^\bmr((0,\Y),\calG^M_{\sigma})$ containing functions that vanish at $y=\Y$:
\[
S_{\{\Y\}}^\bmr((0,\Y),\calG_{\sigma}^M) = \left \{ v_M \in S^\bmr((0,\Y),\calG_{\sigma}^M): v_M(\Y) = 0 \right \}.
\]

Let $\T = \{ K\}$ be a conforming partition of $\bar \Omega$ into simplices $K$. We denote by $\Tr$ a collection of conforming and shape regular meshes that are refinements of an original mesh $\T_0$. For $\T \in \Tr$, we define $h_{\T} = \max \{\diam(K) : K \in \T \}$ and $N = \# \T$, the number of degrees of freedom of $\T$. We introduce the finite element space
\[
S^1_0(\Omega,\T)
= 
\left \{ v_h \in C(\bar \Omega): v_h|_{K} \in \mathbb{P}_{1}(K) \quad 
\forall K \in \T, \ v_h|_{\partial \Omega} = 0 \right \}.
\]

With the meshes $\calG_{\sigma}^M$ and $\T$ at hand, we define $\T_{\Y} = \T \otimes \calG_{\sigma}^M$ and the finite--dimensional \emph{tensor product space}
\begin{equation}\label{eq:TPFE}
\V^{1,\bmr}_{N,M}(\T_{\Y}) := S^1_0(\Omega,\T) \otimes S_{\{ \Y\}}^\bmr((0,\Y),\calG_{\sigma}^M)
\subset \HL(y^{\alpha},\C).
\end{equation}
We write $\V(\T_{\Y})$ if the arguments are clear from the context.

With this discrete setting at hand, we define the finite element approximation $V \in \V(\T_{\Y})$ of the solution $v \in \HL(y^{\alpha},\C_{\Y})$ to problem \eqref{eq:truncated_equation} as follows:
\begin{equation}
V \in \V(\T_{\Y}): \quad \blfa{\C} ( V,W )(q) = d_s \langle f, \tr W \rangle  \quad \forall W \in \V(\T_{\Y}).
\label{eq:discrete_equation}
\end{equation}
The following a priori error estimate can be obtained \cite[Theorem 5]{BMNOSS:17}. 

\begin{theorem}[a priori error estimate]
For arbitrary, fixed $\sigma \in (0,1)$, let $\calG_{\sigma}^M$ be the geometric mesh defined in \eqref{eq:geometric_mesh}, where $\Y \sim |\log h_{\T}|$, with a sufficiently large implied constant, and assume that $M$, the number of elements in $\calG_{\sigma}^M$, satisfies $c_1 M \leq \Y \leq c_2 M$, with absolute constants $c_1$ and $c_2$. If $V \in \V(\T_{\Y})$ denotes the solution to \eqref{eq:discrete_equation} then, there exists a minimal slope $\slope_{\min}$, independent of $h_{\T}$ and $f$, such that for linear degree vectors $\bmr$ with slope $\slope \geq \slope_{\min}$ there holds
\begin{align}\label{8-6-18f}
 \| u - \tr V\|_{\Hs} \lesssim \| \nabla(\ue - V) \|_{L^2(y^{\alpha},\C)} \lesssim h_{\T} \| f \|_{\Ws}.
\end{align}
The hidden constant is independent of $u$, $\ue$, $V$, $f$ and the discretization parameters.
\end{theorem}


\subsection{A fully discrete scheme}
\label{sub:fully_discrete_scheme}
We begin by defining the discrete sets
\begin{equation*}
\mathbb{Q}(\T) = \left\{ Q \in L^{\infty}( \Omega ): Q|_K \in \mathbb{P}_0(K) \quad \forall K \in \T_\Omega \right\},
\quad
\mathbb{Q}_{ad}(\T) = \mathbb{Q}(\T) \cap \calQ,
\end{equation*}
and the discrete coefficient--to--solution operator $F:
\calQ \rightarrow \V(\T_{\Y})$, which associates to an element $q \in \calQ$ the unique discrete solution $V =: F(q)$ of problem \eqref{eq:discrete_equation}. With this operator at hand, we define the discrete reduced cost functional
\begin{equation}
\label{eq:discrete_functional}
\mathcal{D}_{\delta,\rho}(Q):= \frac{1}{2} \| \tr F(Q) - z_{\delta}\|^2_{L^2(\Omega)} + \frac{\rho}{2} \| Q - q^{*}\|^2_{L^2(\Omega)},
\end{equation}
and the following fully discrete approximation of the identification problem \eqref{eq:min_truncated}:
\begin{equation}
\label{eq:min_discrete}
\min_{Q \in \mathbb{Q}_{ad}(\T)} \mathcal{D}_{\delta,\rho}(Q). 
\end{equation}

\begin{lemma}[existence of discrete solutions]
\label{16-5-18f1}
The discrete problem \eqref{eq:min_discrete} has a solution $Q_{\delta,\rho}$.
\end{lemma}

\begin{proof}
Let $(Q_n)_{n \in \mathbb{N}} \subset \mathbb{Q}_{ad}(\T)$ be a minimizing sequence for problem \eqref{eq:min_discrete}. We thus have the existence of a nonrelabeled subsequence of $(Q_n)_{n \in \mathbb{N}}$ that converges, in the $L^2(\Omega)$--norm, to an element $Q$ of the finite dimensional space $\mathbb{Q}_{ad}(\T)$. Standard arguments reveal that the sequence $\left(F(Q_n)\right)_{n \in \mathbb{N}}$ converges to $F(Q)$ in the space $\V(\T_{\Y})$. This concludes the proof.
\end{proof}

In order to present the following result, we define the set
\begin{equation}
\mathcal{I}(u^\dag) := \{q\in \calQ ~|~ U(q) = u^\dag\},
\label{eq:mathcalI}
\end{equation}
In \eqref{eq:mathcalI}, $U$ denotes the coefficient--to--solution operator associated to problem \eqref{eq:fractional_diffusion}. \EO{We assume that $u^\dag$ is an exact state of our identification problem. This immediately yields $\mathcal{I}(u^\dag) \neq \emptyset$.}
 Notice that, in view of Theorem \ref{thm:equivalence}, we have that $\mathcal{I}(u^\dag) = \{q\in \calQ ~|~ \tr E(q) = u^\dag\}$. 

We now introduce the concept of $q^*$-minimum--norm solution for the fractional identification problem \eqref{eq:min_fractional}.

\begin{lemma}[$q^*$-minimum--norm solution]
\label{16-5-18f2}
The optimization problem
\begin{align}\label{16-5-18f3}
\min_{q\in \mathcal{I}(u^\dag)} \|q-q^*\|^2_{L^2(\Omega)}
\end{align}
\EO{attains a solution}, which is called \EO{a} $q^*$-minimum--norm solution of \EO{the identification problem}.
\label{le:q+}
\end{lemma}

\begin{proof}
Let $(q_n)_{n \in \mathbb{N}} \subset \mathcal{I}(u^\dag)$ be a minimizing sequence for problem \eqref{16-5-18f3}, \ie
\begin{align*}
\limn \|q_n-q^*\|^2_{L^2(\Omega)} = \inf_{q\in \mathcal{I}(u^\dag)} \|q-q^*\|^2_{L^2(\Omega)};
\end{align*}
such a sequence does exist by the definition of the infimum. In view of the arguments elaborated in the proof of Theorem \ref{thm:exist}, we deduce the existence of a subsequence $(q_{n_k})_{k \in \mathbb{N}} \subset (q_n)_{n \in \mathbb{N}}$ and an element $\hat q\in \mathcal{Q}$ such that $(q_{n_k})_{k \in \mathbb{N}}$ converges weakly to $\hat q$ in $L^2(\Omega)$ and $\tr \ue(q_{n_k})$ converges strongly to $\tr \ue (\hat q)$ in $L^2(\Omega)$. Since $q_{n_k} \in \mathcal{I}(u^\dag)$, it follows that $\tr \ue(q_{n_k}) = u^\dag$ for all $k\in\mathbb{N}$. Consequently, $\hat q\in \mathcal{I}(u^\dag)$. Now, since $\| \cdot \|_{L^2(\Omega)}$ is a weakly lower semi--continuous function, we conclude that
\[
\|\hat q-q^*\|^2_{L^2(\Omega)} \le \liminf_{k\to\infty} \|q_{n_k}-q^*\|^2_{L^2(\Omega)} = \limn \|q_n-q^*\|^2_{L^2(\Omega)} = \inf_{q\in \mathcal{I}(u^\dag)} \|q-q^*\|^2_{L^2(\Omega)}.
\]
This means that $\hat q$ is a solution of problem \ref{16-5-18f3}. \EO{This concludes the proof.}
\end{proof}

Since it will be instrumental in the analysis that we will perform, we introduce the $L^2(\Omega)$--orthogonal projection operator $\pi_{x'}: L^2(\Omega) \rightarrow \mathbb{Q}(\T)$ as follows: For $K \in \T$ and $q \in L^2(\Omega)$, $\pi_{x'}$ is defined as \cite[Section 1.63]{Guermond-Ern}
\begin{equation}
\label{eq:ortogonal_constant}
\pi_{x'} q|_{K} = \frac{1}{|K|} \int_K q(x') \diff x'.
\end{equation}
Notice that $\pi_{x'} \mathcal{Q} \subset \mathbb{Q}_{ad}(\T_{\Omega})$. In addition, for $1 \leq p \leq \infty$, $\sigma \in (0,1]$, and $q \in W^{\sigma,p}(\Omega)$, we have the following error estimate \cite[Proposition 1.135]{Guermond-Ern}:
\begin{equation}
\label{eq:ortogonal_constant_estimate}
 \| q - \pi_{x'} q \|_{L^p(\Omega)} \lesssim h_{\T}^\sigma | q |_{W^{\sigma,p}(\Omega)}
\end{equation}

In the following result we show that the finite element solutions of problem \eqref{eq:min_discrete} converge to \EO{a} $q^*$-minimum--norm solution of problem \eqref{16-5-18f3}. We stress that the results of Theorem \ref{convergence1} below do not require assuming \eqref{eq:Jcoercive}.

\begin{theorem}[convergence of solutions]
\label{convergence1}
Let $\left(\T_{h_n}\right)_n$ be a sequence of conforming, shape--regular, and  quasi--uniform meshes of $\bar{\Omega}$ and let $h_n:= h_{\T_n}$ be the meshwidth of $\T_n$. \EO{Assume that there exists $0<\gamma < 1$ such that $q^\dag \in W^{\gamma,2}(\Omega)$ with  $q^\dag$ being a solution of \eqref{16-5-18f3}.}
Let $\left(\delta_n\right)_{n \in \mathbb{N}}$ be a sequence in $\mathbb{R}^+$, and consider $\rho_n := \rho_{\delta_n,h_n}$ to be such that
\begin{equation}\label{8-6-18f4}
\rho_n \to 0, \quad \frac{\delta_n^2}{\rho_n} \to 0,
\quad
\mbox{and} 
\quad \frac{ h_n^{2\gamma}  }{\rho_n} \to 0
\enskip\mbox{as}\enskip n \uparrow \infty.
\end{equation}
Let $\left(z_{n}\right)_{n \in \mathbb{N}} = \left(z_{\delta_n}\right)_{n \in \mathbb{N}}$ be a sequence in $L^2(\Omega)$ such that $ \| u^\dag - z_{n} \|_{L^2(\Omega)} \le \delta_n$ and let $Q_n$ be
a minimizer of the following problem (cf.\ \eqref{eq:min_discrete}):
\begin{equation}
\label{eq:minn}
\min_{Q \in \mathbb{Q}_{ad}(\T_{\Omega})} \mathcal{D}_{\delta_n,\rho_n}(Q) := \frac{1}{2} \| \tr F(Q) - z_{n}\|^2_{L^2(\Omega)} + \frac{\rho_n}{2} \| Q - q^{*}\|^2_{L^2(\Omega)}. 
\end{equation}
\EO{Then, there exists a subsequence of $(Q_n)_{n \in \mathbb{N}}$ that converges to a solution of \eqref{16-5-18f3} in $L^2(\Omega)$ as $n \uparrow \infty$. If $q^\dag$ is unique, then convergence holds for the whole sequence.}
\end{theorem}

\begin{proof} 
Since $Q_n$ is optimal for problem \eqref{eq:minn}, we immediately arrive at
\begin{align}\label{24/10:ct1}
\mathcal{D}_{\delta_n,\rho_n}(Q_n)
&\le \frac{1}{2} \| \tr F(\pi_{x'}q^\dag) - z_{n}\|^2_{L^2(\Omega)} + \frac{\rho_n}{2} \| \pi_{x'}q^\dag - q^{*}\|^2_{L^2(\Omega)},
\end{align}
where $\pi_{x'}$ is defined as in \eqref{eq:ortogonal_constant}.

We now proceed to estimate the first term on the right--hand side of \eqref{24/10:ct1}. In view of the estimates \eqref{8-6-18f} and $\| u^\dag - z_{n} \|_{L^2(\Omega)} \le \delta_n$, we obtain that
\begin{multline}
\label{8-6-18f1}
\| \tr F(\pi_{x'}q^\dag) - z_{n}\|_{L^2(\Omega)} 
\leq \| u(\pi_{x'} q^\dag) - \tr F(\pi_{x'}q^\dag)\|_{L^2(\Omega)} + \| u(q^\dag) - u(\pi_{x'} q^\dag)\|_{L^2(\Omega)} 
\\
+ \| z_{n} -u(q^\dag) \|_{L^2(\Omega)} 
\leq Ch_n\| f \|_{\Ws}  + \| u(q^\dag) - u(\pi_{x'} q^\dag)\|_{L^2(\Omega)} + \delta_n.
\end{multline}
We now invoke the extension property \eqref{eq:identity2} and the identity \eqref{eq:norms_are_equal} to arrive at
\begin{multline}
\| u(q^\dag) - u(\pi_{x'} q^\dag)\|_{L^2(\Omega)} = \| \tr \ue (q^\dag) - \tr \ue (\pi_{x'} q^\dag)\|_{L^2(\Omega)} 
\\
\leq \| \tr \ue (q^\dag) - \tr \ue (\pi_{x'} q^\dag)\|_{\Hs} \lesssim  \normC{ \ue (q^\dag) - \ue (\pi_{x'} q^\dag)}.
\label{8-6-18f2}
\end{multline}
It thus suffices to bound $ \normC{ \ue (q^\dag) - \ue (\pi_{x'} q^\dag)}$. Let $ \phi \in \HL(y^\alpha,\C)$. We invoke the problems that $\ue (q^\dag)$ and $ \ue (\pi_{x'} q^\dag)$ solve, combined with the estimates of Remark \ref{rk:first_order}, to arrive at
\begin{align*} 
\blfa{\C} (\ue (q^\dag) - \ue (\pi_{x'} q^\dag),\phi )(q^\dag)
&= \int_\C y^\alpha (\pi_{x'} q^\dag - q^\dag) \ue (\pi_{x'} q^\dag) \phi \diff x' \diff y\\
&\lesssim \| \pi_{x'} q^\dag - q^\dag\|_{L^2(\Omega)} \| \mathtt{e}(\pi_{x'} q^\dag) \|^{\frac{1}{2}}_{L^{\infty}(\Omega)} \|\phi\|_{L^2(y^\alpha,\C)}.
\end{align*}
This implies, in view of the estimate \eqref{eq:ortogonal_constant_estimate}, that 
\[
\normC{ \ue (q^\dag) - \ue (\pi_{x'} q^\dag)} \lesssim h_n^{\gamma} | q^\dag |_{W^{\gamma,2}(\Omega )} \| \mathtt{e}(\pi_{x'} q^\dag) \|^{\frac{1}{2}}_{L^{\infty}(\Omega)}.
\]
This, combined with \eqref{8-6-18f1}--\eqref{8-6-18f2} allow us to conclude
\begin{multline}
\label{8-6-18f3}
\| \tr F(\pi_{x'}q^\dag) - z_{n}\|_{L^2(\Omega)} \lesssim \delta_n + h_n^{\gamma}  | q^\dag |_{W^{\gamma,2}(\Omega)} 
\| \mathtt{e}(\pi_{x'} q^\dag) \|^{\frac{1}{2}}_{L^{\infty}(\Omega)}
\\
+ h_n\| f \|_{\Ws}.
\end{multline}
With the previous estimate at hand, we invoke \eqref{8-6-18f4}--\eqref{24/10:ct1} and conclude that
\begin{align}
\limsup_{n \rightarrow \infty}  \|Q_n - q^*  \|^2_{L^2(\Omega)} \leq \| q^\dag- q^* \|^2_{L^2(\Omega)}.\label{5-12them2}
\end{align}

The estimate \eqref{5-12them2} yields the existence of $\widehat{q} \in \mathcal{Q}$ and a non--relabeled subsequence $ (Q_n)_{n \in \mathbb{N}}$ such that $(Q_n)_{n \in \mathbb{N}}$ converges weakly* to $\widehat{q}$ in $L^\infty(\Omega)$. The arguments developed in the proof of Theorem \ref{thm:exist} thus yield $\|\tr ( \ue (Q_n) - \ue (\widehat{q}) )\|_{L^2(\Omega)} \to 0$ as $n \uparrow \infty$. On the other hand, the trace estimate \eqref{Trace_estimate} and \eqref{8-6-18f}, reveal that
\begin{align*}
\|\tr (\ue (Q_n) - F(Q_n))\|_{L^2(\Omega)} \lesssim \| \nabla(\ue(Q_n) - F(Q_n)) \|_{L^2(y^{\alpha},\C)} \lesssim h_n\| f \|_{\Ws} \to 0
\end{align*}
as $n \uparrow \infty$.
Consequently,
\begin{multline*}
\|\tr \ue (\widehat{q}) - u^\dag\|_{L^2(\Omega)} 
= \limn \|\tr \ue (Q_n) - z_{n}\|_{L^2(\Omega)} \\
\leq \limn \|\tr ( \ue (Q_n) -  F(Q_n) )\|_{L^2(\Omega)} + \limn \|\tr F(Q_n) - z_{n}\|_{L^2(\Omega)} = 0,
\end{multline*}
where we have used that
\begin{multline}
 \limn \frac{1}{2} \|\tr F(Q_n) - z_{n}\|_{L^2(\Omega)} \leq  \limn  \bigg( \frac{1}{2} \| \tr F(\pi_{x'}q^\dag) - z_{n}\|^2_{L^2(\Omega)} 
 \\
 + \frac{\rho_n}{2} \| \pi_{x'}q^\dag - q^{*}\|^2_{L^2(\Omega)} \bigg) = 0,
\end{multline}
which follows from \eqref{24/10:ct1}, \eqref{8-6-18f1}, and \eqref{8-6-18f3}.
We have thus concluded that $\widehat{q} \in \mathcal{I}(u^\dag)$.

Finally, since $q^\dag$ solves \eqref{16-5-18f3}, the fact that $(Q_n)_{n \in \mathbb{N}}$ converges weakly to $\widehat{q}$ in $L^2(\Omega)$, and the estimate \eqref{5-12them2} allow us to conclude that
\begin{align*}
 \|q^\dag - q^* \|^2_{L^2(\Omega)} &\le
\|\widehat{q} - q^* \|^2_{L^2(\Omega)} \le \liminfn
\|Q_n - q^* \|^2_{L^2(\Omega)}
\\
& \le \limsupn
\|Q_n - q^* \|^2_{L^2(\Omega)}
\le \|q^\dag - q^* \|^2_{L^2(\Omega)}.
\end{align*}
\EO{Consequently, $\widehat{q}$ solves \eqref{16-5-18f3}. In addition, $\limn
\|Q_n - q^* \|^2_{L^2(\Omega)} = \|\widehat{q} - q^* \|^2_{L^2(\Omega)}$. This concludes the proof.}
\end{proof}

\subsection{A priori error estimates}
\label{sub:a_priori}
In this section we provide an a priori error analysis for the fully discrete identification problem \eqref{eq:min_discrete} when approximating solutions to \eqref{eq:min_fractional}. To accomplish this task, we follow \cite{MR2536007} and introduce, for $\epsilon >0$, $h_{\T}>0$, and a local solution $r_{\delta,\rho}$ of problem \eqref{eq:min_truncated}, the following minimization problem:
\begin{equation}
\label{eq:min_discrete_epsilon}
\min_{Q \in \mathbb{Q}(r_{\delta,\rho})} \mathcal{D}_{\delta,\rho}(Q),
\end{equation}
where the functional $\mathcal{D}_{\delta,\rho}$ is defined in \eqref{eq:discrete_functional} and $\mathbb{Q}(r_{\delta,\rho}):= \{ Q \in \mathbb{Q}_{ad}(\T) : \| Q - r_{\delta,\rho}\|_{L^2(\Omega)} \leq \epsilon \}$. The next result guarantees existence of solutions for problem \eqref{eq:min_discrete_epsilon}.

\begin{lemma}[existence of solutions]
If $h_{\T}$ is sufficiently small, then problem \eqref{eq:min_discrete_epsilon} has a solution $R_{\delta,\rho}$.
\end{lemma}
\begin{proof}
Define $Q:= \pi_{x'} r_{\delta,\rho}$. A density argument reveals that, if $h_{\T}$ is sufficiently small, then $\| Q - r_{\delta,\rho} \|_{L^2(\Omega)} \leq \epsilon$. Consequently, $\mathbb{Q}(r_{\delta,\rho}) \neq \emptyset$. We thus invoke standard arguments to obtain the desired result. This concludes the proof.
\end{proof}

In the following result, we state a coercivity property for the second order derivatives of the discrete reduced cost functional in a neighborhood of a local solution $r_{\delta,\rho}$ of problem \eqref{eq:min_truncated}.

\begin{lemma}[local coercivity of $\mathcal{D}_{\delta,\rho}''$]
If the assumption \eqref{eq:Jcoercive} holds, $r_{\delta,\rho}$ denotes a local solution for problem \eqref{eq:min_truncated} and  $h_{\T}$ is sufficiently small, then there exist $\epsilon > 0$ such that for every $\hat q \in \mathcal{Q}$ that belongs to neighborhood $ \| \hat q - r_{\delta,\rho}\|_{L^2(\Omega)} \leq \epsilon$, we have that
\begin{equation}
\label{eq:Dcoercive}
\mathcal{D}_{\delta,\rho}''(\hat q)(q,q) \geq \frac{\theta}{4} \| q \|^2_{L^2(\Omega)},
\end{equation}
for all $q \in L^{\infty}(\Omega)$.
\label{le:Dcoercive}
\end{lemma}
\begin{proof}
See \cite[Lemma 4.7]{MR2536007}.
\end{proof}

\begin{lemma}[uniqueness]
Let $\epsilon >0$ be sufficiently small such that \eqref{eq:Dcoercive} holds for $\hat q \in \mathbb{Q}(r_{\delta,\rho})$ and $q \in L^{\infty}(\Omega)$. If $h_{\T}$ is sufficiently small, then problem \eqref{eq:min_discrete_epsilon} admits a unique solution $R_{\delta,\rho}$.
\label{le:uniqueness}
\end{lemma}
\begin{proof}
Let us assume that problem \eqref{eq:min_discrete_epsilon} admits two solutions $R_1$ and $R_2$ in  $\mathbb{Q}(r_{\delta,\rho})$ which are such that $R_1 \neq R_2$. The differentiability properties of  $\mathcal{D}_{\delta,\rho}$ yield
\[
\mathcal{D}_{\delta,\rho}(R_1) = \mathcal{D}_{\delta,\rho}(R_2) +  \mathcal{D}_{\delta,\rho}'(R_2)(R_1 - R_2) +  \frac{1}{2}\mathcal{D}_{\delta,\rho}''(\widetilde R)(R_1 - R_2,R_1 - R_2),
\]
where $\widetilde R = \zeta R_1 + (1-\zeta)R_2$ and $\zeta \in [0,1]$.  Notice that $ \| \tilde R - r_{\delta,\rho} \|_{L^2(\Omega)} 
\leq \epsilon$. 

Since $R_2$ solves \eqref{eq:min_discrete_epsilon}, then $ \mathcal{D}_{\delta,\rho}'(R_2)(R_1 - R_2) \geq 0$. This and an application of the second order optimality condition \eqref{eq:Dcoercive} allow us to conclude that
\[
\mathcal{D}_{\delta,\rho}(R_1) \geq \mathcal{D}_{\delta,\rho}(R_2) + \frac{\theta}{8} \| R_1 - R_2 \|^2_{L^2(\Omega)},
\]
which immediately yields $R_1 \equiv R_2$. This, that is a contradiction with the fact that $R_1 \neq R_2$, concludes the proof.
\end{proof}

In the following  result we show that the solution $R_{\delta,\rho}$ to problem \eqref{eq:min_discrete_epsilon} is a local solution to problem \eqref{eq:min_discrete}. To accomplish this task, we follow arguments elaborated in \cite{ECasas_FTroeltzsch_2002a,MR2536007}.

\begin{lemma}[$R_{\delta,\rho}$ solves problem \eqref{eq:min_discrete}]
Let $\epsilon >0$ be sufficiently small such that \eqref{eq:Dcoercive} holds for $\hat q  \in \mathbb{Q}(r_{\delta,\rho})$ and $q \in L^{\infty}(\Omega)$. Let $R_{\delta,\rho}$ be a solution of \eqref{eq:min_discrete_epsilon} such that $R_{\delta,\rho} \rightarrow r_{\delta,\rho}$ as $h \rightarrow 0$ in $L^2(\Omega)$. If $h_{\T}$ is sufficiently small, then $R_{\delta,\rho}$ is a solution of problem \eqref{eq:min_discrete}.
\label{lm:Rsolves}
\end{lemma}
\begin{proof}
Since  $R_{\delta,\rho}$ solves \eqref{eq:min_discrete_epsilon}, we have that
\begin{equation}
\label{eq:step_aux_R}
\mathcal{D}_{\delta,\rho}(R_{\delta,\rho}) \leq \mathcal{D}_{\delta,\rho}(Q)  \quad \forall Q  \in \mathbb{Q}_{ad}(\T): \| Q - r_{\delta,\rho} \|_{L^2(\Omega)} \leq \epsilon.
\end{equation}
Let $Q \in \mathbb{Q}_{ad}(\T)$ such that $\| Q - R_{\delta,\rho} \|_{L^2(\Omega)} \leq \epsilon/2$. Then, since $R_{\delta,\rho} \rightarrow r_{\delta,\rho}$ as $h_{\T} \rightarrow 0$ in $L^2(\Omega)$, we have, for $h_{\T}$ sufficiently small, that
\[
\| Q - r_{\delta,\rho} \|_{L^2(\Omega)} \leq \| Q - R_{\delta,\rho} \|_{L^2(\Omega)}  + \| R_{\delta,\rho} -  r_{\delta,\rho}\|_{L^2(\Omega)} \leq \epsilon/2 + \epsilon/2 = \epsilon.
\]
Consequently, if $Q \in \mathbb{Q}_{ad}(\T)$ is such that $\| Q - R_{\delta,\rho} \|_{L^2(\Omega)}  \leq \epsilon/2$, then $\| Q - r_{\delta,\rho} \|_{L^2(\Omega)} \leq \epsilon$. In view of \eqref{eq:step_aux_R}, we can thus conclude that
\[
\mathcal{D}_{\delta,\rho}(R_{\delta,\rho}) \leq \mathcal{D}_{\delta,\rho}(Q) \quad \forall Q  \in \mathbb{Q}_{ad}(\T): \| Q - R_{\delta,\rho} \|_{L^2(\Omega)} \leq \epsilon/2.
\]
This proves that $R_{\delta,\rho}$ solves \eqref{eq:min_discrete} and concludes the proof.
\end{proof}

We now derive an a priori error estimate for our fully discrete scheme.

\begin{theorem}[a priori error estimate]
Let $r_{\delta,\rho}$ be a local solution for problem \eqref{eq:min_truncated}. Let $\epsilon$ and $h_{\T}$ sufficiently small such that the result of Lemma \ref{le:uniqueness} hold. If $q^{*} \in H^1(\Omega)$,
we thus have the following error estimate
\[
  \| r_{\delta,\rho} - R_{\delta,\rho} \|_{L^2(\Omega)} \lesssim h_{\T},
 \]
 where the hidden constant is independent of $r_{\delta,\rho}$, $R_{\delta,\rho}$, and $h_{\T}$.
 \label{rhm:error_estimate}
\end{theorem}
\begin{proof}
We proceed in several steps.

\fbox{1} We begin this step by choosing $\epsilon >0$ sufficiently small such that, for $\hat q \in \mathcal{Q}$ that lies in the neighborhood $\|\hat q - r_{\delta,\rho} \|_{L^2(\Omega)} \leq \epsilon$, we have that
\begin{equation}
\label{eq:second_order_R}
  \mathcal{R}_{\delta,\rho}''(\hat q)(q,q) \geq \frac{\theta}{2} \| q \|^2_{L^2(\Omega)} \quad \forall q \in L^{\infty}(\Omega),
\end{equation}
and, for $\hat Q \in \mathbb{Q}_{ad}(\T)$ in the neighborhood $\|\hat Q - r_{\delta,\rho} \|_{L^2(\Omega)} \leq \epsilon$, we have that
\begin{equation}
\label{eq:second_order_D}
 \mathcal{D}_{\delta,\rho}''(\hat Q)(q,q) \geq \frac{\theta}{4} \| q \|_{L^2(\Omega)}^2 \quad \forall q \in L^{\infty}(\Omega).
\end{equation}

Let us now introduce the following optimization problem: 
\begin{equation}
\label{eq:min_discrete_truncated_epsilon}
\min_{Q \in \mathbb{Q}(r_{\delta,\rho})} \mathcal{R}_{\delta,\rho}(Q).
\end{equation}
We recall that $\mathbb{Q}(r_{\delta,\rho}):= \{ Q \in \mathbb{Q}_{ad}(\T) : \| Q - r_{\delta,\rho}\|_{L^2(\Omega)} \leq \epsilon \}$ and that $\mathcal{R}_{\delta,\rho}$ is defined in \eqref{eq:functional_truncated}. Notice that, if $h_{\T}$ is sufficiently small, then \eqref{eq:min_discrete_truncated_epsilon} has a unique solution $Z_{\delta,\rho}$.

A basic application of the triangle inequality allows us to arrive at
\begin{equation}
\| r_{\delta,\rho} - R_{\delta,\rho} \|_{L^2(\Omega)} \leq \| r_{\delta,\rho} -Z_{\delta,\rho}  \|_{L^2(\Omega)} + \|Z_{\delta,\rho} - R_{\delta,\rho} \|_{L^2(\Omega)}.
\label{eq:first_split}
\end{equation}

\fbox{2} We proceed to estimate the term $\| r_{\delta,\rho} -Z_{\delta,\rho}  \|_{L^2(\Omega)}$. To accomplish this task, we set, for $\zeta \in [0,1]$, $\hat q = \zeta  r_{\delta,\rho} + (1-\zeta) Z_{\delta,\rho}$ and notice that 
\[
\|\hat q - r_{\delta,\rho} \|_{L^2(\Omega)} = (1 - \zeta) \| r_{\delta,\rho} - Z_{\delta,\rho} \|_{L^2(\Omega)} \leq (1 - \zeta)\epsilon \leq \epsilon.
\]
We can thus invoke \eqref{eq:second_order_R} with $q = r_{\delta,\rho} -Z_{\delta,\rho}$ and obtain that
\begin{align*}
& \frac{\theta}{2} \|  r_{\delta,\rho} -Z_{\delta,\rho} \|^2_{L^2(\Omega)} \leq  \mathcal{R}_{\delta,\rho}''(\hat q)( r_{\delta,\rho} -Z_{\delta,\rho}, r_{\delta,\rho} -Z_{\delta,\rho})
\\
& = \mathcal{R}_{\delta,\rho}'( r_{\delta,\rho})( r_{\delta,\rho} -Z_{\delta,\rho}) - \mathcal{R}_{\delta,\rho}'( Z_{\delta,\rho})( r_{\delta,\rho} -Z_{\delta,\rho}) 
\\
& = \mathcal{R}_{\delta,\rho}'( r_{\delta,\rho})( r_{\delta,\rho} -Z_{\delta,\rho}) - \mathcal{R}_{\delta,\rho}'( Z_{\delta,\rho})( r_{\delta,\rho} -\pi_{x'} r_{\delta,\rho}) - \mathcal{R}_{\delta,\rho}'( Z_{\delta,\rho})( \pi_{x'} r_{\delta,\rho} -Z_{\delta,\rho});
\end{align*}
$\pi_{x'}$ denotes the $L^2(\Omega)$--orthogonal projection operator introduced in \eqref{eq:ortogonal_constant}.
 
We now invoke the optimality condition \eqref{eq:first_order_truncated} to arrive at $\mathcal{R}_{\delta,\rho}'( r_{\delta,\rho})( r_{\delta,\rho} -Z_{\delta,\rho}) \leq 0$. On the other hand, the optimality condition for problem \eqref{eq:min_discrete_truncated_epsilon}, for $h_{\T}$ sufficiently small, yields  $- \mathcal{R}_{\delta,\rho}'( Z_{\delta,\rho})( \pi_{x'} r_{\delta,\rho} -Z_{\delta,\rho}) \leq 0$. Thus,
\begin{equation*}
\frac{\theta}{2} \|  r_{\delta,\rho} -Z_{\delta,\rho} \|^2_{L^2(\Omega)} \leq - \mathcal{R}_{\delta,\rho}'( Z_{\delta,\rho})( r_{\delta,\rho} -\pi_{x'} r_{\delta,\rho}).
\end{equation*}
Notice that $ \mathcal{R}_{\delta,\rho}'( Z_{\delta,\rho})( r_{\delta,\rho} -\pi_{x'} r_{\delta,\rho}) =  (\mathfrak{r}(Z_{\delta,\rho}) + \rho(Z_{\delta,\rho} -q^*), r_{\delta,\rho} -\pi_{x'} r_{\delta,\rho})_{L^2(\Omega)}$. On the other hand, since $Z_{\delta,\rho} \in \mathbb{Q}(\T)$, \eqref{eq:ortogonal_constant} yields $ \rho(Z_{\delta,\rho} , r_{\delta,\rho} -\pi_{x'} r_{\delta,\rho})_{L^2(\Omega)} = 0$. We can thus use \eqref{eq:ortogonal_constant}, again, to arrive at the estimate
\begin{multline*}
\frac{\theta}{2} \|  r_{\delta,\rho} -Z_{\delta,\rho} \|^2_{L^2(\Omega)} \leq \rho( q^*-\pi_{x'}q^*, r_{\delta,\rho} -\pi_{x'} r_{\delta,\rho})_{L^2(\Omega)} 
\\
- ( \mathfrak{r}(Z_{\delta,\rho}) -\pi_{x'}\mathfrak{r}(Z_{\delta,\rho}) , r_{\delta,\rho} -\pi_{x'} r_{\delta,\rho})_{L^2(\Omega)}.
\end{multline*}
Consequently,
\begin{multline*}
\|  r_{\delta,\rho} - Z_{\delta,\rho} \|^2_{L^2(\Omega)} 
\lesssim 
\|   r_{\delta,\rho} -\pi_{x'} r_{\delta,\rho} \|^2_{L^2(\Omega)}
+
\|    \mathfrak{r}(Z_{\delta,\rho}) -\pi_{x'}\mathfrak{r}(Z_{\delta,\rho}) \|^2_{L^2(\Omega)}
\\
+\|  q^*-\pi_{x'}q^* \|^2_{L^2(\Omega)}
\lesssim
h_{\T}^2 \left(  \| \nabla_{x'} q^*\|^2_{L^2(\Omega)} +  \| \nabla_{x'}  r_{\delta,\rho} \|^2_{L^2(\Omega)} + \| \nabla_{x'}  \mathfrak{r}(Z_{\delta,\rho}) \|^2_{L^2(\Omega)} \right).
\end{multline*}
Notice that, in view of the fact that $q^* \in H^1(\Omega)$, the regularity results of Theorem \ref{thm:estimate_grad_frakp2} allow us to conclude that the norms involved in the right--hand side of the previous expression are uniformly bounded.

\fbox{3} We now estimate $\|Z_{\delta,\rho} - R_{\delta,\rho} \|_{L^2(\Omega)}$ in \eqref{eq:first_split}. To accomplish this task, we set
\[
\hat Q = \zeta R_{\delta,\rho} + (1-\zeta) Z_{\delta,\rho}, \quad \zeta \in [0,1],
\]
and notice that $\|\hat Q - r_{\delta,\rho}\|_{L^2(\Omega)} \leq \zeta \| R_{\delta,\rho} - r_{\delta,\rho}\|_{L^2(\Omega)} +(1- \zeta) \| Z_{\delta,\rho} - r_{\delta,\rho}\|_{L^2(\Omega)} \leq \epsilon$. We can thus invoke the second order optimality condition \eqref{eq:second_order_D} with $q = Z_{\delta,\rho} - R_{\delta,\rho}$ to arrive at
\begin{align*}
& \frac{\theta}{4} \| R_{\delta,\rho} - Z_{\delta,\rho}\|^2_{L^2(\Omega)} \leq  \mathcal{D}_{\delta,\rho}''(\hat Q)( R_{\delta,\rho} - Z_{\delta,\rho},  R_{\delta,\rho} - Z_{\delta,\rho})
\\
& = \mathcal{D}_{\delta,\rho}'(R_{\delta,\rho})( R_{\delta,\rho} -Z_{\delta,\rho}) - \mathcal{D}_{\delta,\rho}'( Z_{\delta,\rho})( R_{\delta,\rho} -Z_{\delta,\rho}).
\end{align*}
We now invoke the first--order optimality condition for problem \eqref{eq:min_discrete} and conclude that $\mathcal{D}_{\delta,\rho}'( R_{\delta,\rho})( R_{\delta,\rho} -Z_{\delta,\rho}) \leq 0$. On the other hand, the first--order optimality condition for \eqref{eq:min_discrete_truncated_epsilon} yields  $\mathcal{R}_{\delta,\rho}'( Z_{\delta,\rho})( R_{\delta,\rho} -Z_{\delta,\rho}) \geq 0$. Thus,
\begin{align*}
 \frac{\theta}{4} \| R_{\delta,\rho} - Z_{\delta,\rho}\|^2_{L^2(\Omega)} 
 \leq 
\mathcal{R}_{\delta,\rho}'( Z_{\delta,\rho})( R_{\delta,\rho} -Z_{\delta,\rho}) - \mathcal{D}_{\delta,\rho}'( Z_{\delta,\rho})( R_{\delta,\rho} -Z_{\delta,\rho}).
\end{align*}
Consequently, we can obtain that $ \| R_{\delta,\rho} - Z_{\delta,\rho}\|_{L^2(\Omega)} \lesssim h_{\T}$.

\fbox{4} The assertion follows from collecting all the estimates we obtained in previous steps.
\end{proof}

We conclude with the following result.

\begin{corollary}[a priori error estimates]
Let $r_{\delta,\rho}$ be a local solution for problem \eqref{eq:min_truncated}. Let $\epsilon$ 
and $h_{\T}$ sufficiently small such that the result of Lemma \ref{le:uniqueness} hold. 
If $q^{*} \in H^1(\Omega)$,
we thus have the existence of a sequence $\{ Q_{\delta,\rho} \}$ of local minima of \eqref{eq:min_discrete} such that
\[
  \| r_{\delta,\rho} - Q_{\delta,\rho} \|_{L^2(\Omega)} \lesssim h_{\T},
 \]
 where the hidden constant is independent of $r_{\delta,\rho}$, $Q_{\delta,\rho}$, and $h_{\T}$.
\end{corollary}
\begin{proof}
The results of Lemma \ref{lm:Rsolves} show that $R_{\delta,\rho}$ solves \eqref{eq:min_discrete}. The desired error estimate thus follows from Theorem \ref{rhm:error_estimate}.
\end{proof}

\begin{remark}[extensions]\rm
\label{rm:extensions}
We discuss a few extensions of this work.

\begin{itemize}
\item \emph{Observations on a subdomain.} Let us consider the case where 
observations of the exact solution $u^{\dag}$ are only available in a subset $\Omega_{\textrm{obs}}$ of  $\Omega$: $z_\delta \in L^2(\Omega_{\textrm{obs}})$. Let us consider a suitable $L^2(\Omega)$-extension of $z_{\delta}$:
\[
\widehat{z}_\delta := 
\begin{cases}
z_\delta &\quad\mbox{in}\quad \Omega_{\textrm{obs}},\\
u^* &\quad\mbox{in}\quad \Omega\setminus \Omega_{\textrm{obs}}
\end{cases},
\]
where $u^*$ is an a priori estimate of the data. With this extension at hand, we can write the following extended minimization problem:
\begin{equation}\label{10-9-18ct1}
\min_{q\in \mathcal{Q}}\widehat{\calJ}_{\delta,\rho}(q),
\quad
\widehat{\calJ}_{\delta,\rho}(q):=
\frac{1}{2} \| \tr E(q) - \widehat{z}_\delta \|^2_{L^2(\Omega)} + \frac{\rho}{2} \| q - q^* \|_{L^2(\Omega)}^2.
\end{equation}
It can be proved that this problem attains a solution $\widehat{q}_{\delta,\rho}$, as in the case with full observations, which is also a minimizer of the original problem
\begin{equation}\label{10-9-18ct2}
\min_{q\in \mathcal{Q}} \widehat{J}_{\delta,\rho}(q),
\quad
\widehat{J}_{\delta,\rho}(q):=
\frac{1}{2} \| U(q) - \widehat{z}_\delta \|_{L^2(\Omega)}^2 + \frac{\rho}{2} \| q - q^* \|_{L^2(\Omega)}^2.
\end{equation}
The design of an efficient solution technique may be of interest.


\item \emph{Boundary observations.} Boundary observations have been considered for coefficient identification problems involving simple PDEs; see \cite{hikaqu18} and references therein. If $\mathcal{L}$ is supplemented with Neumann boundary conditions, the arguments developed in our work could be of use for exploring identification problems when boundary measurements are considered.
\end{itemize}
\end{remark}

\textbf{Acknowledgment.} We would like to thank the editor and a referee for their insightful
comments and suggestions.
\bibliographystyle{plain}
\bibliography{biblio.bib}
\end{document}